\documentclass[12pt,reqno]{amsart}

\usepackage{amsfonts,amsmath, amssymb,amsthm,amscd}
\usepackage[usenames,dvipsnames]{color}
\usepackage{bm}
\usepackage{mathrsfs}
\usepackage{pdfsync}
\usepackage{yfonts}
\usepackage{hyperref}
\usepackage{graphicx,import,calc}
\usepackage[utf8]{inputenc}
\usepackage{latexsym}
\usepackage{epsfig}
\usepackage{subfigure}
\usepackage{stackrel}
\usepackage{dsfont}
\usepackage[left=2.5cm, right=2.5cm, top=3cm, bottom=3cm]{geometry}
\usepackage{enumitem}
\usepackage{setspace}
\usepackage{mathdots}
\usepackage{mathtools}

\usepackage{tikz}
\usepackage{graphicx}
\usetikzlibrary{shapes.multipart}
\usetikzlibrary{patterns}
\usetikzlibrary{shapes.multipart}
\usetikzlibrary{arrows}
\usetikzlibrary{decorations.markings}

\newtheorem{theorem}{Theorem}[section]

\newtheorem{lemma}[theorem]{Lemma}
\newtheorem{proposition}[theorem]{Proposition}
\newtheorem{corollary}[theorem]{Corollary}

\theoremstyle{definition}
\newtheorem{remark}[theorem]{Remark}

\theoremstyle{definition}

\theoremstyle{definition}
\newtheorem{definition}[theorem]{Definition}

 \newcommand{\R}{\mathbb{R}}
\renewcommand{\P}{\mathbb{P}}

\newcommand{\C}{\mathbb{C}}

\newcommand{\Z}{\mathbb{Z}}
\renewcommand{\L}{\mathbb{L}}
\renewcommand{\o}{\mathbb{\omega}}

\renewcommand{\t}{\tilde}
\newcommand{\msf}{\mathsf}
\newcommand{\mc}{\mathcal}

\newcommand{\mrm}{\mathrm}
\newcommand{\mbf}{\mathbf}
\renewcommand{\Im}{\mathfrak{Im}}
\renewcommand{\Re}{\mathfrak{Re}}

\newcommand{\eeq}{\end{equation}}

\newcommand{\K}{\ensuremath{\mathsf{K}}}
\renewcommand{\i}{\mathbf{i}}

\renewcommand{\epsilon}{\varepsilon}

\title{Tracy-Widom asymptotics for a river delta model}
\author[G. Barraquand]{Guillaume Barraquand}
\address{G. Barraquand,
	Columbia University,
	Department of Mathematics,
	2990 Broadway,
	New York, NY 10027, USA.}
\email{barraquand@math.columbia.edu}
\author[M. Rychnovsky]{Mark Rychnovsky}
\address{M. Rychnovsky, Columbia University,
	Department of Mathematics,
	2990 Broadway,
	New York, NY 10027, USA.}
\email{mrychnov@gmail.com}

\begin{document}
\begin{abstract}
We study an oriented first passage percolation model for the evolution of a river delta. This model is exactly solvable and occurs as the low temperature limit of the beta random walk in random environment. We analyze the asymptotics of an exact formula from \cite{RWRE} to show that, at any fixed positive time, the width of a river delta of length $L$ approaches a constant times $L^{2/3}$ with Tracy-Widom GUE fluctuations of order $L^{4/9}$. This result can be rephrased in terms of particle systems. We introduce an exactly solvable particle system on the integer half line and show that after running the system for only finite time the particle positions have Tracy-Widom fluctuations.
\end{abstract}
\maketitle

\section{Model and results}

\subsection{Introduction}

First passage percolation was introduced in 1965 to study a fluid spreading through a random environment \cite{hammersley1965first}. This model has motivated many tools in modern probability, most notably Kingman's sub-additive ergodic theorem (see the review \cite{fppbook} and references therein); it has attracted attention from mathematicians and physicists alike due to the simplicity of its definition, and the ease with which fascinating conjectures can be stated. 

\smallskip

The Kardar-Parisi-Zhang (KPZ) universality class has also become a central object of study in recent years \cite{universality}. Originally proposed to explain the behavior of growing interfaces in 1986 \cite{originalKPZ}, it has grown to include many types of models including random matrices, directed polymers, interacting particle systems, percolation models, and traffic models. Much of the success in studying these has come from the detailed analysis of a few exactly solvable models of each type. 

\smallskip

We study an exactly solvable model at the intersection of percolation theory and KPZ universality: Bernoulli-exponential first passage percolation (FPP). Here is a brief description (see Definition \ref{model} for a more precise definition). Bernoulli-exponential FPP models the growth of a river delta beginning at the origin in $\Z_{\geq 0}^2$ and growing depending on two parameters $a,b>0$. At time $0$, the river is a single up-right path beginning from the origin chosen by the rule that whenever the river reaches a new vertex it travels north with probability $a/(a+b)$ and travels east with probability $b/(a+b)$ (thick black line in Figure \ref{fig:littlepicture}). The line with slope $a/b$ can be thought of as giving the direction in which the expected elevation of our random terrain decreases fastest. 

\smallskip

\begin{figure}
\begin{center}
		\begin{tikzpicture}[scale=0.4]
	\draw (0,0) node[below left]{{\footnotesize $(0,0)$}};
	\clip(-.2,-.2) rectangle(14.4, 14.4);
	\draw[black!5!] (0,0) grid(16,16);
	
	\draw[line width=.7mm, black!15!]  (4,3) -- (5,3) -- (6,3) -- (7,3) -- (7,4) -- (8,4) -- (8,5) -- (9,5) -- (9,6) -- (10,6) -- (11,6) -- (11,7) -- (11,8) -- (11,9);
	\draw[line width=.9mm, black!25!] (3,7) -- (3,8) -- (3,9) -- (4,9) -- (5,9) -- (6,9) -- (6,10) -- (7,10) --   (7,11) -- (8,11) -- (9,11) -- (10,11) -- (10,12) -- (11,12) -- (12,12) -- (12,13) -- (12,14) -- (13,14) -- (13,15) -- (14,15);
	\draw[line width=1.2mm, black!50!] (2, 3) -- (3,3) -- (4,3) -- (4,4) -- (4,5) -- (5,5) -- (5,6) -- (6,6) -- (6,7);
	\draw[line width=1.6mm] (0,0) -- (0,1) -- (0,2) -- (1,2) -- (1,3) -- (2,3) -- (2,4) -- (2,5) -- (2,6) -- (3,6) -- (3,7) -- (4,7) --(5,7) --(6,7) --(7,7) -- (8,7) -- (8,8) -- (9,8) -- (9,9) -- (10,9) -- (11,9) -- (12,9) -- (12,10) -- (13,10) -- (14,10) -- (14,11) -- (14,12) -- (14,13) -- (15,13);
	\end{tikzpicture}
\end{center}
	\caption{A sample of the river delta (Bernoulli-exponential FPP percolation cluster) near the origin. The thick black random walk path corresponds to the river (percolation cluster) at time $0$. The other thinner and lighter paths correspond to tributaries added to the river delta (percolation cluster) at later times.}
	\label{fig:littlepicture}
\end{figure}
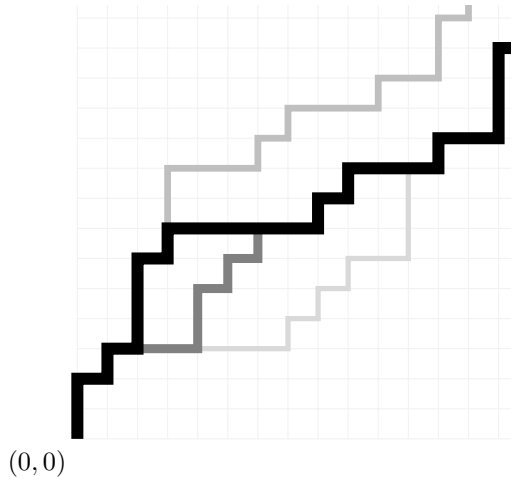

As time passes the river erodes its banks creating forks. At each vertex which the river leaves in the rightward (respectively upward) direction, it takes an amount of time distributed as an exponential random variable with rate $a$ (resp. $b$) for the river to erode through its upward (resp. rightward) bank. Once the river erodes one of its banks at a vertex, the flow at this vertex branches to create a tributary (see gray paths in Figure \ref{fig:littlepicture}). The path of the tributary is selected by the same rule as the path of the time $0$ river, except that when the tributary meets an existing river it joins the river and follows the existing path. The full path of the tributary is added instantly when the river erodes its bank.

\smallskip

In this model the river is infinite, and the main object of study is the set of vertices included in the river at time $t$, i.e. the percolation cluster. We will also refer to the shape enclosed by the outermost tributaries at time $t$ as the river delta (see Figure \ref{largescaleBEFPP} for a large scale illustration of the river delta).

\smallskip

The model defined above can also be seen as the low temperature limit of the beta random walk in random environment (RWRE) model \cite{RWRE}, an exactly solvable model in the KPZ universality class. Bernoulli-exponential FPP is particularly amenable to study because an exact formula for the distribution of the percolation cluster's upper border (Theorem \ref{exact} below) can be extracted from an exact formula for the beta RWRE \cite{RWRE}. We perform an asymptotic analysis on this formula to prove that at any fixed time, the width of the river delta satisfies a law of large numbers type result with fluctuations converging weakly to the Tracy-Widom GUE distribution (see Theorem \ref{main theorem}). Our law of large numbers result was predicted in \cite{RWRE} by taking a heuristic limit of \cite[Theorem 1.19]{RWRE}; we present this non-rigorous computation in Section \ref{main result}. We also give other interpretations of this result. In Section \ref{interpretations} we introduce an exactly solvable particle system and show that the position of a particle at finite time has Tracy-Widom fluctuations. 

\begin{figure} 	
\begin{center}
		\includegraphics[scale=.45]{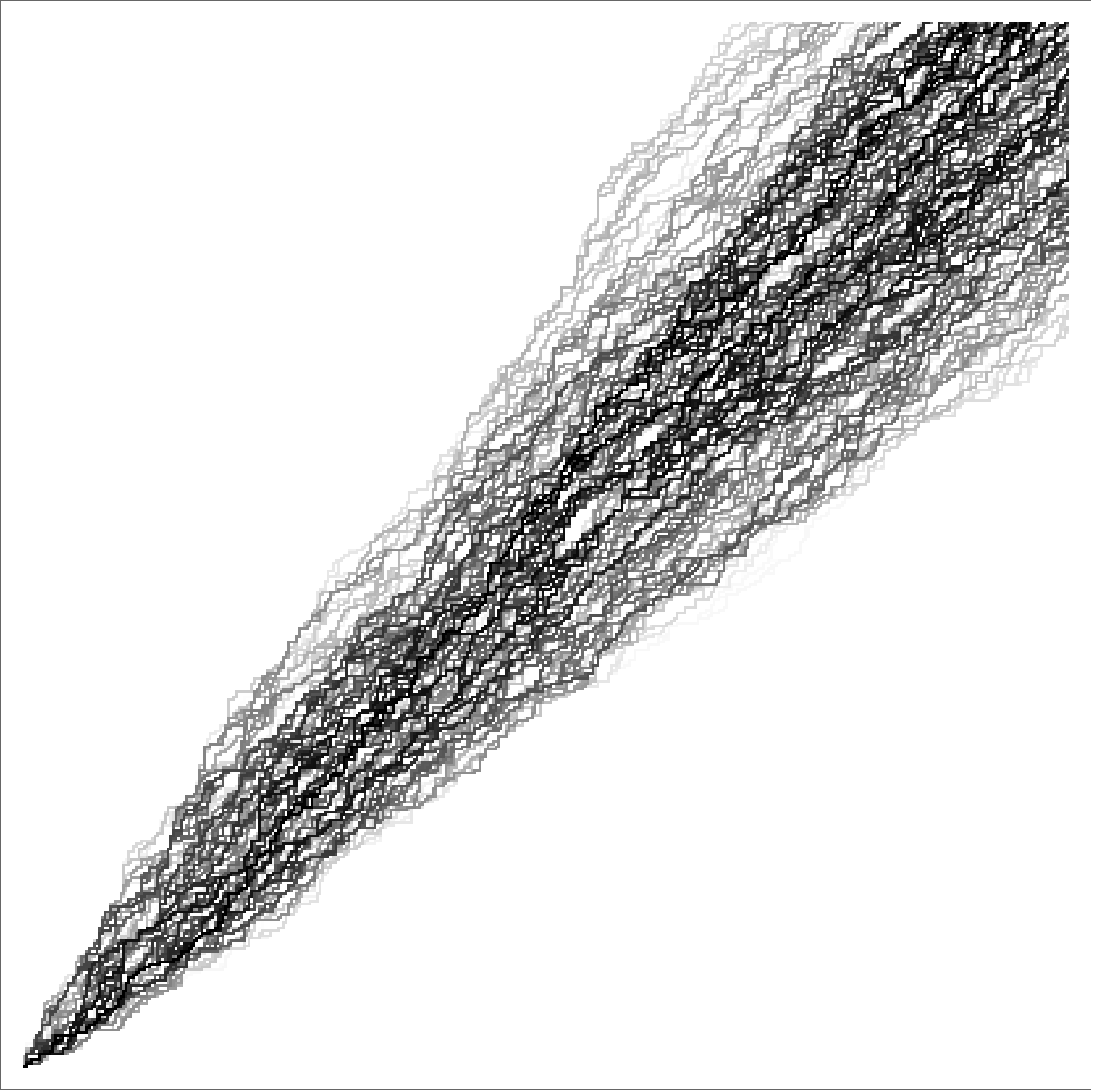}
\end{center}
	\caption{The percolation cluster for $400 \times 400$ Bernoulli-exponential FPP at time $1$ with $a=b=1$. Paths occurring earlier are shaded darker, so the darkest paths occur near $t=0$ and the lightest paths occur near $t=1$.}
	\label{largescaleBEFPP}
\end{figure}

\subsection{Definition of the model}
We now define the model more precisely in terms of first passage percolation following \cite{RWRE}.
\begin{definition}[Bernoulli-exponential first passage percolation] \label{model}
	Let $E_e$ be a family of independent exponential random variables indexed by the edges $e$ of the lattice $\Z_{\geq0}^2$. Each $E_e$ is distributed as an exponential random variable with parameter $a$ if $e$ is a vertical edge, and with parameter $b$ if $e$ is a horizontal edge. Let $(\zeta_{i,j})$ be a family of independent Bernoulli random variables with parameter $b/(a+b)$. We define the passage time $t_e$ of each edge $e$ in the lattice $\Z_{\geq 0}^2$ by
	$$t_e=\begin{cases}
	\zeta_{i,j} E_e \qquad \text{if $e$ is the vertical edge $(i,j) \to (i,j+1)$},\\
	(1-\zeta_{i,j}) E_e \qquad \text{if $e$ is the horizontal edge $(i,j) \to (i+1,j)$}.\\
	\end{cases}$$
\end{definition}

We define the point to point passage time $T^{\mathrm{PP}}(n,m)$ by 

$$T^{\mathrm{PP}}(n,m)=\min_{\pi:(0,0) \to (n,m)} \sum_{e \in \pi} t_e.$$
where the minimum is taken over all up-right paths from $(0,0)$ to $(n,m)$. We define the percolation cluster $C(t)$, at time $t$, by

$$C(t)=\left\{ (n,m): T^{\mathrm{PP}}(n,m) \leq t \right\}.$$
At each time $t$, the percolation cluster $C(t)$ is the set of points visited by a collection of up-right random walks in the quadrant $\Z_{\geq 0}^2$. $C(t)$ evolves in time as follows:

\begin{itemize}[leftmargin=*]
	
	\item At time $0$, the percolation cluster contains all points in the path of a directed random walk starting from $(0,0)$, because at any vertex $(i,j)$ we have passage time $0$ to either $(i,j+1)$ or $(i+1,j)$ according to the independent Bernoulli random variables $\zeta_{i,j}$. 
	\item At each vertex $(i,j)$ in the percolation cluster $C(t)$, with an upward (resp. rightward) neighbor outside the cluster, we add a random walk starting from $(i,j)$ with an upward (resp. rightward) step to the percolation cluster with exponential rate $(a)$ (resp. $b$). This random walk will almost surely hit the percolation cluster after finitely many steps, and we add to the percolation cluster only those points that are in the path of the walk before the first hitting point (see Figure \ref{fig:littlepicture}).
	
\end{itemize}
Define the height function $H_t(n)$ by 
\begin{equation}
H_t(n)=\sup \{ m \in \Z_{\geq 0}| T^{\mathrm{PP}}(n,m) \leq t)\},
\label{eq:defheighfunction}
\end{equation}
so that $(n,H_t(n))$ is the upper border of $C(t)$. 

\subsection{History of the model and related results}
Bernoulli-exponential FPP was first introduced in \cite{RWRE}, which introduced an exactly solvable model called the beta random walk in random environment (RWRE) and studied Bernoulli-exponential FPP as a low temperature limit of this model (see also the physics works \cite{thiery2015integrable,thiery2016exact} further studying the Beta RWRE and some variants). The beta RWRE was shown to be exactly solvable in \cite{RWRE} by viewing it as a limit of $q$-Hahn TASEP, a Bethe ansatz solvable particle system introduced in \cite{pov13}. The $q$-Hahn TASEP was further analyzed in \cite{borodin2015spectral,cor14,qhahnboson}, and was recently realized as a degeneration of the higher spin stochastic six vertex model \cite{aggarwal2017dynamical,borodin2017family,borodin2018higher,corwin2016stochastic}, so that Bernoulli-exponential FPP fits as well in the framework of stochastic spin models.

Tracy-Widom GUE fluctuations were shown in \cite{RWRE} for Bernoulli-exponential FPP (see Theorem \ref{heuristic theorem}) and for Beta RWRE. In the Beta RWRE these fluctuations occur in the quenched large deviation principle satisfied by the random walk and for the maximum of many random walkers in the same environment.

The connection to KPZ universality was strengthened in subsequent works. In \cite{corwin-gu} it was shown that the heat kernel for the time reversed Beta RWRE converges to the stochastic heat equation with multiplicative noise. In \cite{balazs-rassoul} it was shown using a stationary version of the model that a Beta RWRE conditioned to have atypical velocity has wandering exponent $2/3$ (see also \cite{chaumont2017fluctuation}), as expected in general for directed polymers in $1+1$ dimensions. The stationary structure of Bernoulli-exponential FPP was computed in \cite{thieryzerotemp} (In \cite{thieryzerotemp} Bernoulli-exponential FPP is referred to as the Bernoulli-exponential polymer).

The first occurrence of the Tracy-Widom distribution in the KPZ universality class dates back to the work of Baik, Deift and Johansson on longest increasing subsequences of random permutations \cite{baik1999distribution} (the connection to KPZ class was explained in e.g. \cite{prahofer2000universal}) and the work of Johansson on TASEP \cite{johansson2000shape}. 
In the past ten years, following Tracy and Widom's work on ASEP \cite{tracy2008integral,tracy2008fredholm,asep} and Borodin and Corwin's Macdonald processes \cite{borodin2014macdonald}, a number of exactly solvable $1+1$ dimensional models in the KPZ universality class have been analyzed asymptotically. Most of them can be realized as more or less direct degenerations of the higher-spin stochastic six-vertex model. This includes particle systems such as exclusion processes (q-TASEP \cite{borodin2014duality,barraquand2014phase,qtasep,orr2017stochastic} and  other models \cite{barraquand2016q,baik2017facilitated,ghosalpush,qhahnboson}), directed polymers (\cite{FreeEnergyCorwin,borodin2013log,borodin2015height,corwin2014strict,krishnan2016tracy,o2014tracy}), and the stochastic six-vertex model \cite{aggarwal2016phase,aggarwal2016current,barraquand2017stochastic,borodin2016stochastic,borodin2016asep}.

\subsection{Main result}\label{main result}

The study of the large scale behavior of passage times $T^{\mathrm{PP}}(n,m)$ was initiated in \cite{RWRE}. At large times, the fluctuations of the upper border of the percolation cluster (described by the height function $H_t(n)$) has GUE Tracy-Widom  fluctuations on the scale $n^{1/3}$.  
\begin{theorem} [{\cite[Theorem 1.19]{RWRE}}]\label{heuristic theorem} 
	Fix parameters $a,b>0$. For any $\theta>0$ and $x\in \R$, 
	\begin{equation}
	\lim_{n \to \infty} \P \left(\frac{H_{\tau(\theta)n} - \kappa(\theta)n}{\tilde\rho(\theta) n^{1/3}}\leq x  \right)=F_{\mathrm{GUE}}(x),\label{eq:limittheoremlargetime}
	\end{equation}
	where $F_{\textrm{GUE}}$ is the GUE Tracy-Widom distribution (see Definition \ref{tracywidomdist}) and $\kappa(\theta)$, $\tau(\theta)$, $\tilde\rho(\theta)= \frac{\kappa'(\theta)}{\tau'(\theta)} \rho(\theta)$ are functions defined in \cite{RWRE} by 
	\begin{align*}
	\kappa(\theta)&:=\frac{ \frac{1}{\theta^2}-\frac{1}{(a+\theta)^2}}{\frac{1}{(a+\theta)^2}-\frac{1}{(a+b+\theta)^2}},\\
	\tau(\theta)&:=\frac{1}{a+\theta}-\frac{1}{\theta}+\kappa(\theta) \left( \frac{1}{a+\theta}-\frac{1}{a+b+\theta} \right)=\frac{a(a+b)}{\theta^2(2a+b+2\theta)},\\
	\rho(\theta)&:=\left[ \frac{1}{\theta^3}-\frac{1}{(a+\theta)^3}+\kappa(\theta) \left(\frac{1}{(a+b+\theta)^3}-\frac{1}{(a+\theta)^3} \right) \right]^{1/3}.
	\end{align*}
\end{theorem}
Note that as $\theta$ ranges from $0$ to $\infty$, $\kappa(\theta)$ ranges from $+\infty$ to $a/b$ and $\tau(\theta)$ ranges from $+\infty$ to $0$. 
\begin{remark}
	In \cite{RWRE} the limit theorem is incorrectly  stated as $$\lim_{n \to \infty} \P \left(\frac{\min_{i\leq n}T^{\textrm{PP}}(i, \kappa(\theta) n)-\tau(\theta) n}{\rho(\theta) n^{1/3}} \leq x \right)=F_{\mathrm{GUE}}(x),$$
	but following the proof in \cite[Section 6.1]{RWRE}, we can see that the inequality and the sign of $x$ should be reversed. Further, we have reinterpreted the limit theorem in terms of height function $H_t(n)$ instead of passage times  $T^{\textrm{PP}}(n,m)$ using the relation \eqref{eq:defheighfunction}. 
\end{remark}

In this paper, we are interested in the fluctuations of $H_t(n)$ for large $n$ but fixed time $t$. 
Let us scale $\theta$ in \eqref{eq:limittheoremlargetime} above as 
$$ \theta = \left( \frac{na(a+b)}{2t} \right)^{1/3},$$
so that 
$$ \tau(\theta)n = t +O(n^{-1/3}).$$
Let us introduce constants
\begin{equation}\lambda= \left( \frac{a(a+b)}{2t}\right)^{1/3}, \quad d=\frac{3a(a+b)}{2b \lambda}, \quad \sigma=\left(\frac{3a(a+b)\lambda}{2b^3}\right)^{1/3}. \label{constants} \end{equation}
Then, we have the approximations 
\begin{align*}
\kappa(\theta)n &= \frac{a}{b}n + dn^{2/3} + o(n^{4/9}),\\
\tilde\rho(\theta)n^{1/3} &= \sigma n^{4/9} + o(n^{4/9}). 
\end{align*}
Thus, formally letting $\theta$ and $n$ go to infinity in \eqref{eq:limittheoremlargetime} suggests that for a fixed time $t$, it is natural to scale the height function as
$$H_{t}(n) =  \frac{a}{b}n + dn^{2/3} + \sigma n^{4/9} \chi_n,$$
and study the asymptotics of the sequence of random variables $\chi_n$.

Our main result is the following. 
\begin{theorem} \label{main theorem}
	Fix parameters $a,b>0$.  For any $t>0$ and  $x \in \mathbb{R}$,
	$$\lim_{n \to \infty} \P\left(\frac{H_t(n)-\frac{a}{b} n-d n^{2/3}}{\sigma n^{4/9}} \leq x\right) = F_{\textrm{GUE}}(x),$$
	where $F_{\textrm{GUE}}$ is the GUE Tracy-Widom distribution.
\end{theorem}

Note that the heuristic argument presented above to guess the scaling exponents and the expression of constants $d$ and $\sigma$ is  not rigorous, since Theorem \ref{heuristic theorem} holds for fixed $\theta$. Theorem \ref{heuristic theorem}  could be extended without much effort to a weak convergence uniform in $\theta$ for $\theta$ varying in a fixed compact subset of $(0,+\infty)$. However the case of  $\theta$ and $n$ simultaneously  going to infinity requires more careful analysis. Indeed, for $\theta$ going to infinity very fast compared to $n$, Tracy-Widom fluctuations would certainly disappear as this would correspond to considering the height function at time $\tau(\theta)n\approx 0$, that is a simple random walk having Gaussian fluctuations on the $n^{1/2}$ scale. We explain in the next section how we shall prove Theorem \ref{main theorem}. 

The scaling exponents in Theorem 2 might seem unusual, although the preceding heuristic computation explains how they result from rescaling a model which has the usual KPZ scaling exponents. A similar situation occurs for scaling exponents of the height function of directed last passage percolation in thin rectangles \cite{baik2005gue,UniversalPropertyCloseToAxis} and for the free energy of directed polymers  \cite{UniversalityRectanglesCorwin} under the same limit. 

\subsection{Outline of the Proof} 
Recall that given an integral kernel $\msf{K}: \mathbb{C}^2 \to \mathbb{C}$, its Fredholm determinant is defined as
$$\det(1+\msf{K})_{L^2(\mathcal{C})}:=\frac{1}{2 \pi \i}\sum_{n=0}^{\infty} \frac{1}{n!} \int_{\mc{C}^n} \det[\msf{K}(x_i,x_j)]_{i,j=1}^n dx_1...dx_n.$$
To prove Theorem \ref{main theorem} we begin with the following Fredholm determinant formula for $\mathbb{P}(H_t(n)<m)$, and perform a saddle point analysis.
\begin{theorem}[{\cite[Theorem 1.18]{RWRE}}]\label{exact}
	$$\P(H_t(n) < m) =\det(I-\mathsf{K}_{n})_{\L^2(\mathcal{C}_0)},$$
	where $\mathcal{C}_0$ is a small positively oriented circle containing $0$ but not $-a-b$, and $\mathsf{K}_{n}: \L^2(\mathcal{C}_0) \to \L^2(\mathcal{C}_0)$ is defined by its integral kernel 
	\begin{align}
	\mathsf{K}_{n}(u,u') &=\frac{1}{2\pi \i} \int_{1/2 -i \infty}^{1/2+i \infty} \frac{e^{ts}}{s} \frac{g(u)}{g(s+u)} \frac{ds}{s+u-u'}, \quad \textrm{where} \label{eq:defK}\\
	g(u) &=\left(\frac{a+u}{u}\right)^n \left(\frac{a+u}{a+b+u}\right)^m \frac{1}{u}. 
	\end{align}
\end{theorem}
\begin{remark}
	Note that \cite[Theorem 1.18]{RWRE} actually states $\P(H_t(n)<m) = \det(I+\mathsf{K}_{n})_{\L^2(\mathcal{C}_0)}$, instead of $\det(I-\mathsf{K}_{t,n})_{\L^2(\mathcal{C}_0)}$ due to a sign mistake.
\end{remark}
This result was proved in \cite{RWRE} by taking a zero-temperature limit of a similar formula for the Beta RWRE obtained using the Bethe ansatz solvability of $q$-Hahn TASEP and techniques from \cite{borodin2014macdonald,borodin2014duality}.  The integral \eqref{eq:defK} above is oscillatory and does not converge absolutely, but we may deform the contour so that it does. We will justify this deformation in Section $2.2$. 

Theorem \ref{main theorem} is proven in Section 2 by applying steep descent analysis to $\det(1-\mathsf{K}_n)$, however the proofs of several key lemmas are deferred to later sections. The main challenge in proving Theorem \ref{main theorem} comes from the fact that, after a necessary change of variables $\o=n^{-1/3}u$, the contours of the Fredholm determinant are being pinched between poles of the kernel $\mathsf{K}_n$ at $\o=0$ and $\o=\frac{-a-b}{n^{1/3}}$ as $n \to \infty$. In order to show that the integral over the contour near $0$ does not affect the asymptotics, we prove bounds for $\mathsf{K}_n$ near $0$, and carefully choose a family of contours $\mc{C}_n$ on which we can control the kernel. This quite technical step is the main goal of Section 3. Section 4 is devoted to bounding the Fredholm determinant expansion of $\det(1-\mathsf{K}_n)_{L^2(\mathcal{C}_n)},$ in order to justify the use of dominated convergence in Section 2. 

\subsection{Other interpretations of the model} \label{interpretations}

There are several equivalent interpretations of Bernoulli-exponential first passage percolation. We will present the most interesting here. 

\subsubsection{A particle system on the integer line}
The height function of the percolation cluster $H_t(n)$ is equivalent to the height function of an interacting particle system we call geometric jump pushTASEP, which generalizes pushTASEP (the $R=0$ limit of PushASEP introduced in \cite{BorodinFerrari}) by allowing jumps of length greater than 1. This model is similar to Hall-Littlewood pushTASEP introduced in \cite{ghosalpush}, but has a slightly different particle interaction rule. 

\begin{definition}[Geometric jump pushTASEP] Let $\mathrm{Geom}(q)$ denote a geometric random variable with $\P(\mathrm{Geom}(q)=k)=q^k(1-q)$. Let $1 \leq  p_1(t) < p_2(t) < ...<p_i(t)<...$ be the positions of ordered particles in $\Z_{\geq 1}$. At time $t=0$ the position $n \in \Z_{\geq 0}$ is occupied with probability $b/(a+b)$. Each particle has an independent exponential clock with parameter $a$, and when the clock corresponding to the particle at position $p_i$ rings, we update each particle position $p_j$ in increasing order of $j$ with the following procedure. ($p_i(t-)$ denotes the position of particle $i$ infinitesimally before time $t$.)
	
	\begin{itemize}
		\item If $j<i$, then $p_j$ does not change. 
		\item $p_i$ jumps to the right so that the difference $p_i(t)-p_i(t-)$ is distributed as $1+\mathrm{Geom}(a/(a+b))$
		\item If $j>i$, then
		\begin{itemize}
			\item If the update for $p_{j-1}(t)$ causes $p_{j-1}(t) \geq p_{j}(t-)$, then $p_{j}(t)$ jumps right so that $p_{j}(t)-p_{j-1}(t)$ is distributed as $1+\mathrm{Geom}(a/(a+b))$.  
			\item Otherwise $p_j$ does not change.
			\item All the geometric random variables in the update procedure are independent.
		\end{itemize} 
	\end{itemize}
\end{definition}

\begin{figure}
\begin{center}
		\includegraphics[scale=.4]{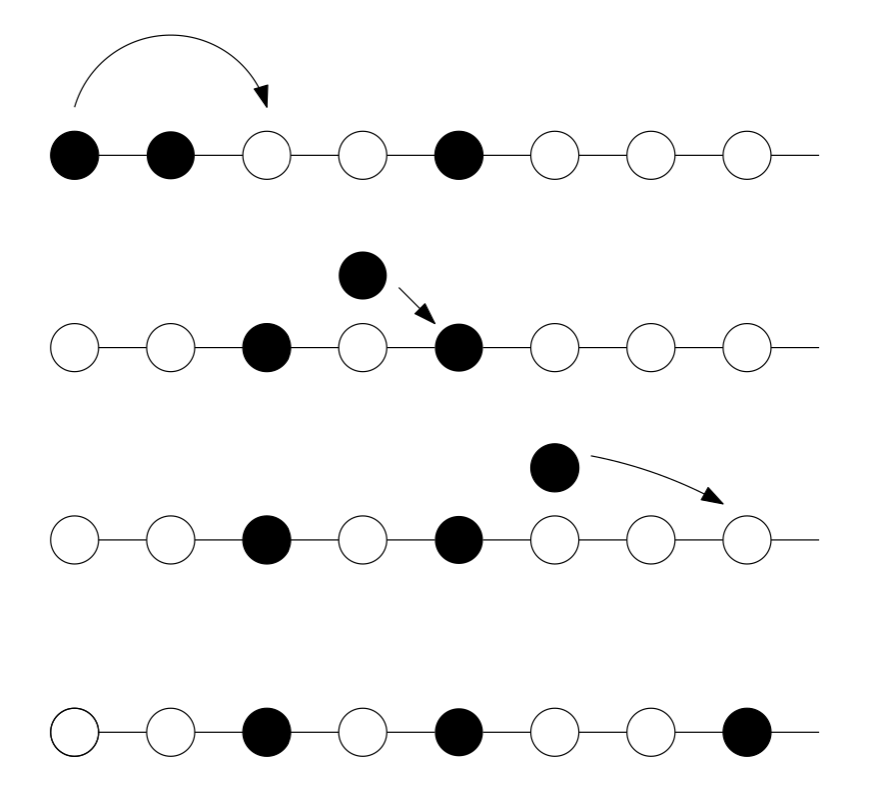}
\end{center}
	\caption{This figure illustrates a single update for geometric jump pushTASEP. The clock corresponding to the leftmost particle rings, activating the particle. The first particle jumps 2 steps pushing the next particle and activating it. This particle jumps 1 step pushing the rightmost particle and activating it. The rightmost particle jumps 3 steps, and all particles are now in their original order, so the update is complete.}
	\label{particlesystem}
\end{figure}

Another way to state the update rule is that each particle jumps with exponential rate a, and the jump distance is distributed as $1 + \mathrm{Geom}(a/(a+b))$. When a jumping particle passes another particle, the passed particle is pushed a distance $1+\mathrm{Geom}(a/(a+b))$ past the jumping particle's ending location (see Figure \ref{particlesystem}).

The height function $\overline{H}_t(n)$ at position $n$ and time $t$ is the number of unoccupied sites weakly to the left of $n$. If we begin with the distribution of $(n,H_t(n))$ in our percolation model, and rotate the first quadrant clockwise $45$ degrees, the resulting distribution is that of $(n,\overline{H}_t(n))$. The horizontal segments in the upper border of the percolation cluster correspond to the particle positions, thus
$$H_t(n)=p_t(n)-n=\sup\{k: \overline{H}_t(n+k) \geq k\}.$$
A direct translation of Theorem \ref{main theorem} gives:
\begin{corollary} \label{cor:particle} Fix parameters $a,b>0$. For any $t>0$ and $x \in \mathbb{R}$,
	$$\lim_{n \to \infty} \mathbb{P}\left(\frac{p_t(n)-\left(\frac{a+b}{b}\right) n-d n^{2/3}}{\sigma n^{4/9}} \leq x\right)=F_{\mathrm{GUE}}(x),$$
	where $F_{\mathrm{GUE}}(x)$ is the Tracy-Widom GUE distribution.
\end{corollary}
To the authors knowledge Corollary \ref{cor:particle} is the first result in interacting particle systems showing Tracy-Widom fluctuations for the position of a particle at finite time.

\subsubsection{Degenerations}

If we set $b=1, t'= t/a,$ and $a \to 0$, then in the new time variable $t'$ each particle performs a jump with rate  1 and with probability going to 1, each jump is distance 1, and each push is distance 1. This limit is pushTASEP on $\Z_{\geq 0}$ where every site is occupied by a particle at time $0$. Recall that in pushTASEP, the dynamics of a particle are only affected by the (finitely many) particles to its left, so this initial data makes sense.

We can also take a continuous space degeneration. Let $x$ be the spatial coordinate of geometric jump pushTASEP, and let $\exp(\lambda)$ denote an exponential random variable with rate $\lambda$. Choose a rate $\lambda>0$, and set $b=\frac{\lambda}{n}, x'=x/n, a=\frac{n-\lambda}{n}$, and let $n \to \infty$. Then our particles have jump rate $\frac{n-\lambda}{n} \to 1$, jump distance $\frac{\mathrm{Geom}(1-\lambda/n)}{n} \to \exp(\lambda)$, and push distance $\frac{\mathrm{Geom}(1-\lambda/n)}{n} \to \exp(\lambda)$. This is a continuous space version of pushTASEP on $\R_{\geq 0}$ with random initial conditions such that the distance between each particle position $p_i$ and its rightward neighbor $p_{i+1}$ is an independent exponential random variable of rate $\lambda$. Each particle has an exponential clock, and when the clock corresponding to the particle at position $p_i$ rings, an update occurs which is identical to the update for geometric jump pushTASEP except that each occurrence of the random variable $1 +\mathrm{Geom}(a/(a+b))$ is replaced by the random variable $\exp(\lambda).$

\subsubsection{A benchmark model for travel times in a square grid city}

The first passage times of Bernoulli-exponential FPP can also be interpreted as the minimum amount of time a walker must wait at streetlights while navigating a city \cite{stoplight}. Consider a city, whose streets form a grid, and whose stoplights have i.i.d exponential clocks. The first passage time of a point $(n,m)$ in our model has the same distribution as the minimum amount of time a walker in the city has to wait at stoplights while walking $n$ streets east and $m$ streets north. Indeed at each intersection the walker encounters one green stoplight with zero passage time and one red stoplight at which they must wait for an exponential time. Note that while the first passage time is equal to the waiting time at stoplights along the best path, the joint distribution of waiting times of walkers along several paths is different from the joint passage times along several paths in Bernoulli-exponential FPP.

\subsection{Further directions} \label{questions}
Bernoulli-exponential FPP has several features that merit further investigation. From the perspective of percolation theory, it would be interesting to study how long it takes for the percolation cluster to contain all vertices in a given region, or how geodesics from the origin coalesce as two points move together.

From the perspective of KPZ universality, it is natural to ask: what is the correlation length of the upper border of the percolation kernel, and what is the joint law of the topmost few paths.

Under diffusive scaling limit, the set of coalescing simple directed random walks originating from every point of $\Z^2$ converges to the Brownian web \cite{fontes2002brownian,fontes2004brownian}. Hence the set of all possible tributaries in our model converges to the Brownian web. One may define a more involved set of coalescing and branching random walks which converges to a continuous object called the Brownian net (\cite{newman2010marking}, \cite{sun2008brownian}, see also the review \cite{schertzer2015brownian}). Thus, it is plausible that there exist a continuous limit of Bernoulli-Exponential FPP where tributaries follow Brownian web paths and branch at a certain rate at special points of the Brownian web used in the construction of the Brownian net.

After seeing Tracy-Widom fluctuations for the edge statistics it is natural to ask whether the density of vertices inside the river along a cross section is also connected to random matrix eigenvalues and whether a statistic of this model converges to the positions of the second, third, etc. eigenvalues of the Airy point process.

\subsection{Notation and conventions}
We will use the following notation and conventions.
\begin{itemize}
	\item $B_{\epsilon}(x)$ will denote the open ball of radius $\epsilon>0$ around the point $x$. 
	\item $\Re[x]$ will denote the real part of a complex number $x$, and $\Im[x]$ denotes the imaginary part. 
	\item $\mathcal{C}$ and $\mathcal{\gamma}$ with any upper or lower indices will always denote an integration contour in the complex plane. $\mathsf{K}$ with any upper or lower indices will always represent an integral kernel. A lower index like $\mathcal{\gamma}_r$, $\mathcal{C}_n$, or $\mathsf{K}_n$ will usually index a family of contours or kernels. An upper index such as $\mathcal{\gamma}^{\epsilon}$, $\mathcal{C}^{\epsilon}$, or $\mathsf{K}^{\epsilon}$ will indicate that we are intersecting our contour with a ball of radius $\epsilon$, or that the integral defining the kernel is being restricted to a ball of radius $\epsilon$. 
	
\end{itemize}

\subsection*{Acknowledgements}
The authors thank Ivan Corwin for many helpful discussions and for useful comments on an earlier draft of the paper. The authors thank an anonymous reviewer for detailed and helpful comments on the manuscript. G. B. was partially supported
by the NSF grant DMS:1664650. M. R. was partially supported by the Fernholz Foundation's ``Summer Minerva Fellow'' program, and also received summer support from Ivan Corwin's NSF grant DMS:1811143.

\section{Asymptotics}

\subsection{Setup}
The steep descent method is a method for finding the asymptotics of an integral of the form
$$I_M=\int_{\mathcal{C}} e^{Mf(z)}dz,$$
as $M \to \infty$, where $f$ is a holomorphic function and $\mathcal{C}$ is an integration contour in the complex plane. The technique is to find a critical point $z_0$ of $f$, deform the contour $\mathcal{C}$ so that it passes through $z_0$ and $\Re[f(z)]$ decays quickly as $z$ moves along the contour $\mathcal{C}$ away from $z_0$. In this situation $e^{Mf(z_0)}/e^{Mf(z)}$ has exponential decay in $M$. We use this along with specific information about our $f$ and $\mathcal{C}$, to argue that the integral can be localized at $z_0$, i.e. the asymptotics of $\int_{\mathcal{C} \cap B_{\epsilon}(z_0)} e^{M f(z)} dz$ are the same as those of $I_M$. Then we Taylor expand $f$ near $z_0$ and show that sufficiently high order terms do not contribute to the asymptotics. This converts the first term of the asymptotics of $I_M$ into a simpler integral that we can often evaluate.

In Section $2.1$ we will manipulate our formula for $\P(h(n) < m)$, and find a function $f_1$ so that the kernel $\msf{K}_{n}$ can be approximated by an integral of the form $\int_{\lambda+\i \mathbb{R}} e^{n^{1/3}[f_1(z)-f_1(\o)]} dz$. Approximating $\msf{K}_{n}$ in this way will allow us to apply the steep descent method to both the integral defining $\msf{K}_n$ and the integrals over $\mathcal{C}_0$ in the Fredholm determinant expansion. 

\medskip
For the remainder of the paper we fix a time $t>0$, and parameters $a,b>0$. All constants arising in the analysis below depend on those parameters $t,a,b$, though we will not recall this dependency explicitly for simplicity of notation.  
\medskip

We also fix henceforth  
\begin{equation} m=\left\lfloor \frac{a}{b} n +d n^{2/3}+n^{4/9} \sigma x \right\rfloor. \label{m}\end{equation}
We consider $\msf{K}_n$ and change variables setting $\tilde{z}=s+u$, $d\tilde{z}=ds$ to obtain
$$\tilde{\mathsf{K}}_n(u,u')=\frac{1}{2 \pi \i} \int_{1/2+u-\i \infty}^{1/2+u+\i \infty} \frac{e^{t(\tilde{z}-u)}}{(\tilde{z}-u)(\t{z}-u')} \frac{g(u)}{g(\t{z})} d\t{z}.$$
In the following lemma, we change our contour of integration in the $\t{z}$ variable so that it does not depend on $u$.
\begin{lemma}
	For every fixed $n$,
	$$\t{\mathsf{K}}_n(u,u')=\frac{1}{2 \pi \i} \int_{n^{1/3} \lambda+\i \mathbb{R}} \frac{e^{t(\t{z}-u)}}{(\t{z}-u)(\t{z}-u')} \frac{g(u)}{g(\t{z})} d\t{z}.$$ 
\end{lemma}

\begin{proof}
	Choose the contour $\mathcal{C}_0$ to have radius $0<r<\min[1/4,\lambda]$. This choice of $r$ means that we do not cross $\mc{C}_0$ when deforming the contour $1/2+u+\i \mathbb{R}$ to $\lambda+\i \mathbb{R}$. In this region $K$ is a holomorphic function, so this deformation does not change the integral provided that for $M$ real,
	
	$$\frac{1}{2 \pi \i}\int_{1/2+u+\i M}^{n^{1/3}\lambda+\i M} \frac{e^{t(\tilde{z}-u)}}{(\tilde{z}-u)(\tilde{z}-u')} \frac{g(u)}{g(\tilde{z})} d\tilde{z} \xrightarrow[M \to \pm \infty]{} 0.
	$$
	This integral converges to $0$ because for all $\tilde{z} \in [n^{1/3} \lambda-\i M,1/2+u-\i M] \cup [n^{1/3} \lambda+\i M,1/2+u+\i M]$ we have 
	$$\left|\frac{1}{(\tilde{z}-u)(\tilde{z}-u')g(\tilde{z})}\right| \sim \frac{1}{M},$$ 
	as $M \to \infty$.
	
\end{proof}

Set 
$$\tilde{h}_n(z)=-n \log \left( \frac{a+z}{z}\right )-m \log \left(\frac{a+z}{a+b+z} \right), \qquad \text{so that} \qquad e^{\tilde{h}_n(z)}=\frac{z}{g(z)}.$$
Then 
$$\mathsf{K}_n(u,u')=\frac{1}{2 \pi \i} \int_{n^{1/3} \lambda+\i \R} \frac{e^{t\t{z}+\tilde{h}_n(\t{z})}}{e^{tu+\t{h}_n(u)}} \frac{\t{z}}{u} \frac{d\t{z}}{(\t{z}-u)(\t{z}-u')}.$$
Now perform the change of variables 
$$z=n^{-1/3}\t{z}, \o=n^{-1/3} u, \o'=n^{-1/3} u'.$$
If we view our change of variables as occuring in the Fredholm determinant expansion, then due to the $d\o_i$s, we see that scaling all variables by the same constant does not change the Fredholm determinant $\det(1-\K_n)_{L^2(\mathcal{C})}$. Thus our change of variables gives
$$\mathsf{K}_n(\o,\o')=\frac{1}{2 \pi \i} \int_{\lambda+\i \R}  \frac{e^{n^{1/3} t(z-\o)}}{(z-\o)(z-\o')} e^{h_n(z)-h_n(\o)} \frac{z}{\o} dz$$
where 
$$h_n(z)=\t{h}_n(n^{1/3}z)=-n \log \left(\frac{a+n^{1/3}z}{n^{1/3} z} \right) -m \log \left( \frac{a+n^{1/3}z}{a+b+n^{1/3}z} \right).$$

\begin{remark} The contour for $\o$, $\o'$ becomes $n^{-1/3} \mc{C}_0$ after the change of variables, but $\K_n(\o,\o')$ is holomorphic in most of the complex plane. Examining of the poles of the integrand for $\K_n(\o,\o')$, we see that we can deform the contour for $\o,\o'$ in any way that does not cross the line $\lambda+ \i \R$, the pole at $-(a+b)/n^{1/3}$, or the pole at $0$, without changing the Fredholm determinant $\det(I-\K_n)_{L^2(n^{-1/3}\mathcal{C}_0)}$.
	
\end{remark}
Taylor expanding the logarithm in the variable $n$ gives
$$h_n(z)=-n^{1/3} \left(\frac{a(a+b)}{2z^2}-\frac{bd}{z}\right)-n^{1/9} \left( \frac{-b \sigma x}{z} \right)+r_n(z).$$
Here $r_n(z) = \mathcal O(1)$ in a sense that we make precise in Lemma \ref{chaos control}. The kernel can be rewritten as 
\begin{multline*}
\mathsf{K}_n(\o,\o')= \\\frac{1}{ 2 \pi \i} \int_{\lambda+\i \R} \frac{\exp(n^{1/3} (f_1(z)-f_1(w))+n^{1/9}(f_2(z)-f_2(\o))+(r_n(z)-r_n(\o)))}{(z-\o)(z-\o')} \frac{z}{\o} dz
\end{multline*}
where
\begin{equation} f_1(z)=tz-\frac{a(a+b)}{2z^2}+\frac{bd}{z}, \qquad
f_2(z)=\frac{b \sigma x}{z}. \label{f def}
\end{equation}
We have approximated the kernel as an integral of the form  $\int e^{n^{1/3} [f_1(z)-f_1(\o)]}dz$. To apply the steep-descent method, we want to understand the critical points of the function $f_1$. We have
\begin{equation}
f_1'(z)=t+\frac{a(a+b)}{z^3}-\frac{db}{z^2},\qquad
f_1''(z)=-\frac{3a(a+b)}{z^4}+\frac{2bd}{z^3}, \qquad
f_1'''(z)=\frac{12a(a+b)}{z^5}-\frac{6bd}{z^4}.
\end{equation}
Where $a,b$ are the parameters associated to the model. Let the constant $\lambda$ be as defined in (\ref{constants}), then $0=f_1'(\lambda)=f_1''(\lambda)=0$, and 

$$f_1'''(\lambda)=\frac{3a(a+b)}{\lambda^5}=2\left(\frac{ b \sigma}{\lambda^2}\right)^3=2\left(\frac{-f_2'(\lambda)}{x}\right)^3 ,$$
is a positive real number. $\sigma$ is defined in equation (\ref{constants}).

Recall the definition of the Tracy-Widom GUE distribution, which governs the largest eigenvalue of a gaussian hermitian random matrix. 
\begin{definition} \label{tracywidomdist} The Tracy-Widom distribution's distribution function is defined as $F_{\textrm{GUE}}(x)=\det(1-\mathsf{K}_{\mathrm{Ai}})_{L^2(x,\infty)}$, where $K_{Ai}$ is the Airy kernel,
	\begin{align}\mathsf{K}_{\mathrm{Ai}}(s,s')=\frac{1}{2\pi \i} \int_{e^{-2 \pi i/3} \infty}^{e^{2 \pi \i /3} \infty} d\o \frac{1}{2\pi \i} \int_{e^{-\pi \i/3} \infty}^{e^{\pi \i/3} \infty}dz \frac{e^{z^3/3-zs}}{e^{\o^3/3-\o s'}} \frac{1}{(z-\o)}.\nonumber \end{align}
\end{definition}
In the above integral the two contours do not intersect. We can think of the inner integral following the contour $(e^{-\pi \i /3} \infty, 1] \cup (1, e^{\pi \i/3} \infty)$, and the outer integral following the contour $(e^{-2\pi \i/3} \infty, 0] \cup (0, e^{2\pi \i/3} \infty)$. Our goal through the rest of the paper is to show that the Fredholm determinant $\det(I-\msf{K}_n)$ converges to the Tracy-Widom distribution as $n \to \infty$.
\subsection{Steep descent contours} 

\begin{definition}
	We say that a path $\gamma: [a,b] \to \mathbb{C}$ is steep descent with respect to the function $f$ at the point $x= \gamma(0)$ if $\frac{d}{dt} \Re[f(\gamma(t))]>0$ when $t>0$, and $\frac{d}{dt} \Re[f(\gamma(t))]<0$ when $t<0$. 
\end{definition}

We say that a contour $\mathcal{C}$ is steep descent with respect to a function $f$ at a point $x$, if the contour can be parametrized as a path satisfy the above definition. Intuitively this statement means that as we move along the contour $\mathcal{C}$ away from the point $x$, the function $f$ is strictly decreasing. 

In this section we will find a family of contours $\mc{\gamma}_r$ for the variable $z$ and so that $\mc{\gamma}_r$ is steep descent with respect to $\Re[f_1(z)]$ at the point $\lambda,$ and study the behavior of $\Re[f_1]$. The contours $\mathcal{C}_n$ for $\o$ are constructed in Section \ref{Cn}.

\begin{lemma}
	\label{steep descent}
	
	The contour $\lambda+\i \mathbb{R}$ is steep descent with respect to the function $\Re[f_1]$ at the point $\lambda$. 
	
\end{lemma}

\begin{proof}
	
	We have that
	
	$$\frac{d}{dy} \Re[f_1(\lambda+\i y)]=-\Im[f_1'(\lambda+\i y)]=-\Im\left[t+\frac{a(a+b)}{(\lambda+\i y)^3}-\frac{bd}{\lambda+\i y}\right].$$
	Now using the relation $2 b d \lambda= 3a(a+b)$ and computing gives
	
	$$\frac{d}{dy} \Re[f_1(\lambda+\i y)]=\frac{-4a(a+b)y^3}{(\lambda^2+y^2)^3}.$$
	This derivative is negative when $y>0$ and positive when $y<0$.
	
\end{proof}

\begin{figure}[h]
\begin{center}
		\includegraphics[scale=0.4]{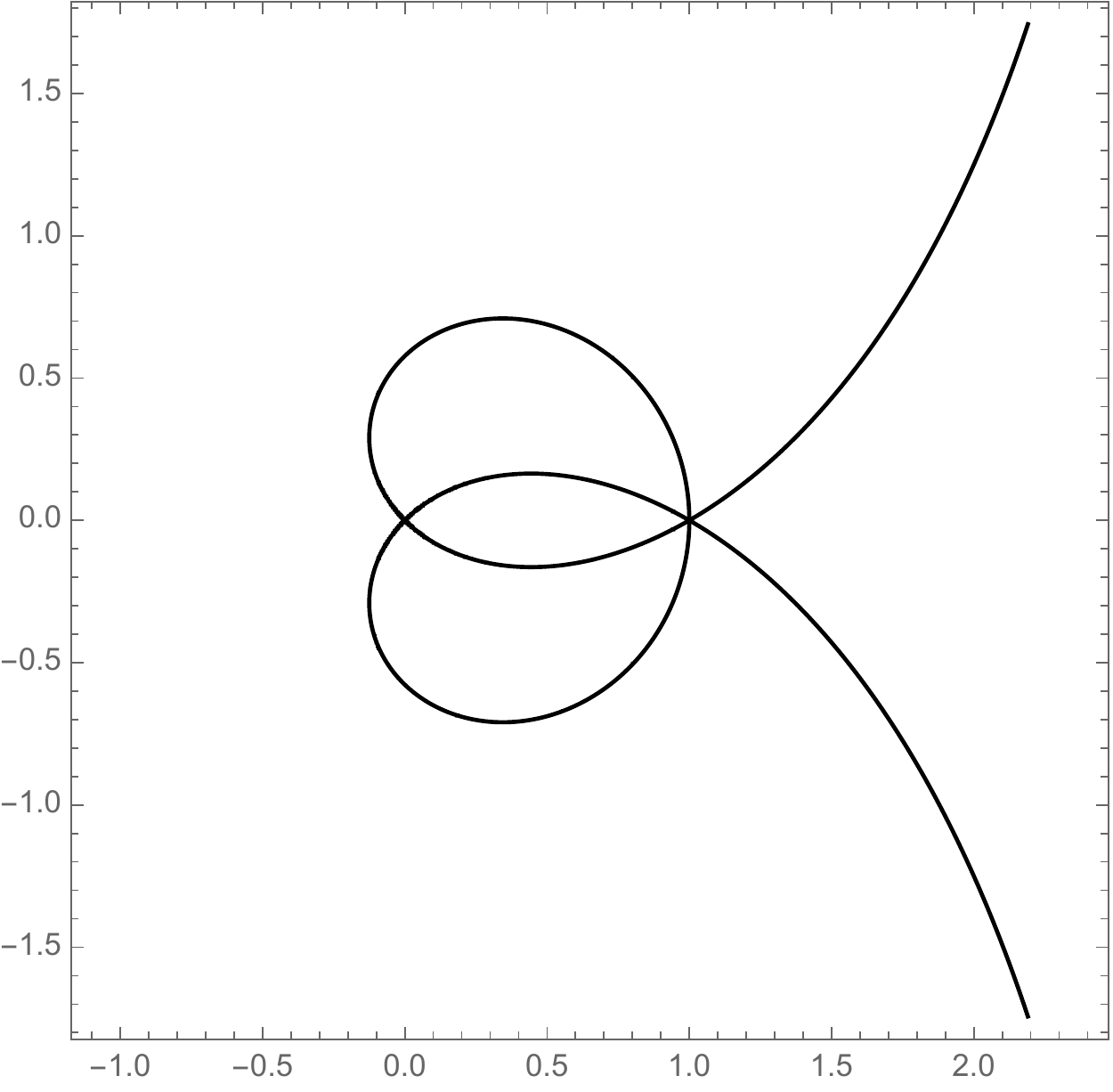} 
\end{center}
	\caption{The level lines of the function $\Re[f_1(z)]$ at value $\Re[f_1(\lambda)]$. In this image we take $a=b=t=1$.}
	\label{contour}
\end{figure}
Now we describe the contour lines of $\Re[f_1(z)]$ seen in Figure \ref{contour}. $\Re[f_1]$ is the real part of a holomorphic function, so its level lines are constrained by its singularities, and because the singularities are not too complicated, we can describe its level lines. The contour lines of the real part of a holomorphic function intersect only at critical points and poles and the number of contour lines that intersect will be equal to the degree of the critical point or pole. We can see from the Taylor expansion of $f_1$ at $\lambda$, that there will be $3$ level lines intersecting at $\lambda$ with angles $\pi/6,\pi/2$, and $5\pi/6$. From the form of $f_1$, we see that there will be $2$ level lines intersecting at $0$ at angles $\pi/4$ and $3\pi/4$, and that a pair of contour lines will approach $\i \infty$ and $-\i \infty$ respectively with $\Re[z]$ approaching $f_1(\lambda)/t$. This shows that, up to a noncrossing continuous deformation of paths, the lines in Figure $\ref{contour}$ are the contour lines $\Re[f_1(z)]=f_1(\lambda)$. We can also see that on the right side of the figure, $tz$ will be the largest term of $\Re[f_1(z)]$, so our function will be positive. This determines the sign of $\Re[f_1(z)]$ in the other regions. 

Our contour $\lambda+ \i \mathbb{R}$ is already steep descent, but we will deform the tails, so that we can use dominated convergence in the next section.
\begin{definition} For any $r>0$, define the contour $\mathcal{\gamma}_r=(e^{-2 \pi \mathbf{i}/3} \infty, \lambda-r\mathbf{i}) \cup [\lambda-r\mathbf{i},\lambda+r\mathbf{i}] \cup (\lambda+r\mathbf{i},e^{2\pi \mathbf{i}/3}\infty)$ and $\mathcal{\gamma}_r^{\epsilon}=\mathcal{\gamma}_r \cap B_{\epsilon}(\lambda).$ These contours appear in Figure $\ref{contours}$.
	
\end{definition}

\begin{figure}
\begin{center}
	\begin{tikzpicture}[scale=0.9]
\draw[ultra thick] (0,-1)--(0,1);
\draw[dashed] (0,1)--(0,2);
\draw[dashed] (0,1)--(0,-2);
\draw[dashed] (0,-2) -- +(-120:2) node[below]{$e^{-2 \pi \mathbf{i}/3} \infty$};
\draw[dashed] (0,2)-- +(120:2) node[above]{$e^{2 \pi \mathbf{i}/3} \infty$};

\node[right] at (-2,0) {$0$};
\node[right] at (0,2) {$\lambda+\i r$};
\node[right] at (0,-2) {$\lambda-\i r$};
\node[right] at (0,0) {$\lambda$};
\node[right] at (0,1) {$\lambda+\i \epsilon$};
\node[right] at (0,-1) {$\lambda- \i \epsilon$};

\draw[fill] (0,0) circle [radius=.06];
\draw[fill] (0,2) circle [radius=.06];
\draw[fill] (0,-2) circle [radius=.06];
\draw[fill] (0,1) circle [radius=.06];
\draw[fill] (0,-1) circle [radius=.06];
\draw[fill] (-2,0) circle [radius=.06];

\end{tikzpicture}
\end{center}
\caption{The contour $\mathcal{\gamma}_r$ is the infinite piecewise linear curve formed by the union of the vertical segment and the two semi infinite rays, oriented from bottom to top. The bold portion of this contour near $\lambda$ is $\mathcal{\gamma}_r^{\epsilon}$.}
\label{contours}
\end{figure}
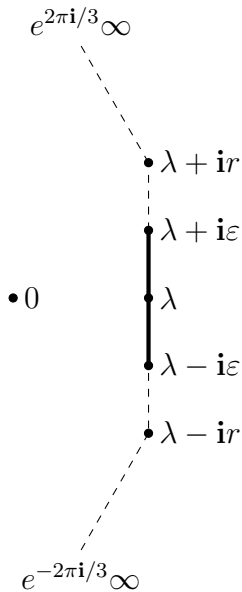

Because for any fixed $n$, we have $e^{h_n(z)} \to 1$ as $|z| \to \infty$, $\frac{z}{\o(z-\o)(z-\o')}$ has linear decay in $z$, and $e^{n^{1/3}t(z-\o)}$ has exponential decay in $z$, we can deform the vertical contour $\lambda+ \i \mathbb{R}$ to the contour $\mathcal{\gamma}_r$. Thus
$$\mathsf{K}_n(\o,\o')=\int_{\mathcal{\gamma}_r} \frac{e^{n^{1/3}t(z-\o)}}{(z-\o)(z-\o')} e^{h_n(z)-h_n(\o)} \frac{z}{\o} dz.$$

The function $\Re[f_1]$ is still steep descent on the contour $\mc{\gamma}_r$ with respect to the point $\lambda$. Lemma \ref{steep descent} shows that $\Re[f_1]$ is steep descent on the segment $[\lambda-r\mathbf{i},\lambda+r\mathbf{i}]$, and on $(e^{-2 \pi \mathbf{i}/3} \infty, \lambda-r\mathbf{i}) \cup (\lambda+r\mathbf{i},e^{2\pi \mathbf{i}/3} \infty)$ we inspect $f'_1(z)$ and note that for $z$ sufficiently large, the constant term $t$ dominates the other terms. Because our paths are moving in a direction with negative real component the contour $\mc{\gamma}_r$ is steep descent. 

Up to this point we have been concerned with contours being steep descent with respect to $\Re[f_1]$, but the true function in our kernel is $\exp(n^{1/3}t(z-\o)+h_n(z)-h_n(\o))$. To show that $\mc{\gamma}_r$ is steep descent with respect to this function, we will need to control the error term $n^{1/3}tz+h_n(z)-n^{1/3}f_1(z)=n^{1/9}f_2(z)+r_n(z)$. The following lemma gives bounds on this error term away from $z=0$. 

\begin{lemma}
	\label{chaos control}
	For any $N,\epsilon>0$ there is a constant $C$ depending only on $\epsilon, N$ such that
	\begin{equation}|f_2(\o)| \leq  C \text{  and  } |r_n(\o)| \leq C, \label{control} \end{equation}
	for all $n \geq N,$ and $\o \geq \frac{|a+b|+\epsilon}{N^{1/3}}$. \\
	Similarly for any $\delta>0$, there exists $N_{\delta}$  and $C'$ depending only on $\delta$, such that 
	\begin{equation} |f'_2(\o)| \leq C' \text{  and  } |r_n'(\o)| \leq C', \label{derivative control} \end{equation}
	for all $n \geq N_{\delta}$, and $\o$ satisfying $|\o| \geq \delta.$
\end{lemma}
Lemma \ref{chaos control} is proved in Section \ref{Cn}.


\medskip

At this point we have a contour $\mc{\gamma}_r$ for the variable $z$, which is steep descent with respect to $\Re[f_1]$. We want to find a suitable contour for $\o$. The following lemma shows the existence of such a contour $\mathcal{C}_n$, where property $(c)$ below takes the place of being steep descent. This lemma is fairly technical and its proof is the main goal of Section \ref{Cn}. To see why observe that the function $n^{1/3}f_1(\o)$ does not approximate $n^{1/3}t \o-h_n(\o)$ well when $\o$ is near $0$. The fact that the contribution near $0$ is negligible is nontrivial because the function $n^{1/3}t \o -h_n(\o)$ has poles at $0$ and $\frac{-a-b}{n^{1/3}}$, and our contour $\mc{C}_n$ is being pinched between them; we will use Lemma \ref{C bound} to show that the asymptotics of $\det(1-\K_n)_{L^2(\mathcal{C}_n)}$ are not affected by these poles

\begin{lemma} \label{C bound} There exists a sequence of contours $\{\mc{C}_n\}_{n \geq N}$ such that:
	\begin{itemize}
		
		\item[(a)] For all $n$, the contour $\mathcal{C}_n$ encircles $0$ counterclockwise, but does not encircle $(-a-b)n^{-1/3}$. 
		
		\item[(b)] $\mc{C}_n$ intersects the point $\lambda$ at angles $-\pi/3$ and $-2\pi/3$. 
		
		\item[(c)] For all $\epsilon>0$, there exists $\eta, N_{\epsilon}>0$ such that for all $n>N_{\epsilon}$, $\o \in \mathcal{C}_n \setminus \mathcal{C}_n^{\epsilon}$ and $z \in \mc{\gamma}_r$, we have
		
		$$\Re[n^{1/3}t(z-\o)+h_n(z)-h_n(\o)] \leq -n^{1/3}\eta,$$
		where $\mc{C}_n^{\epsilon}=\mc{C}_n \cap B_{\epsilon}(\lambda).$
		
		\item[(d)] There is a constant $C$ such that for all $\o \in \mathcal{C}_n$, 
		
		$$\Re[n^{1/3}t(\lambda-\o)+h_n(\lambda)-h_n(\o)] \leq n^{1/9}C.$$
		
	\end{itemize}
\end{lemma}

The next lemma allows us to control $\Re[n^{1/3}tz+h_n(z)]$ on the contour $\mc{\gamma}_r$.

\begin{lemma}
	\label{gamma bound}
	For all $\epsilon>0$, and for sufficiently large $r$, there exists $C,N_{\epsilon}>0$, such that for all $\o \in \mc{C}_n$, and $z \in \mathcal{\gamma}_r \setminus \mathcal{\gamma}_r^{\epsilon}$, then 
	$$\Re[h_n(z)-h_n(\o)+n^{1/3}t(z-\o)] \leq -n^{-1/3} C.$$
\end{lemma}
\begin{proof}
	We have already shown that $\mc{\gamma}_r$ is steep descent with respect to $f_1(z)$. 
	
	By Lemma \ref{chaos control}, $|r_n| \leq C, |f_2| \leq Cn^{1/9}$ away from $0$. We have
	\begin{align}h_n(z)-h_n(\o)+n^{1/3}t(z-\o)=&n^{1/3}(f_1(z)-f_1(\o))+n^{1/9}(f_2(z)-f_2(\o))+(r_n(z)-r_n(\o))\nonumber\\
	\leq n^{1/3}(f_1(z)-&f_1(\o))+n^{1/9}C+C \leq n^{1/3}(f_1(z)-f_1(\o)+\delta), \nonumber \end{align}
	for any sufficiently small $\delta>0$. Because $f_1(z)$ is decreasing as we move away from $\lambda$, we have  
	$$n^{1/3}tz+h_n(z)<n^{1/3}t \lambda+h_n(\lambda)+Cn^{1/9}.$$
	Thus by $\ref{chaos control}$, we have that for all $\epsilon>0$ there exists $C$ such that for $z \in \mathcal{\gamma}_r \setminus \mathcal{\gamma}_r^{\epsilon}$,
	$$\Re[h_n(z)-h_n(\lambda)+n^{1/3}t(z-\lambda)] \leq -n^{1/3} C.$$
	By Lemma \ref{C bound} (d), we have
	$$\Re[h_n(\lambda)-h_n(\o)+n^{1/3}t(\lambda-\o)] \leq n^{1/9}C,$$
	for $\o \in \mathcal{C}_n.$
	This completes the proof
\end{proof}

\subsection{Localizing the integral}\label{localize}
In this section we will use Lemma \ref{C bound} and Lemma \ref{gamma bound} to show that the asymptotics of $\det(1-\mathsf{K}_n)_{L^2(\mathcal{C}_n)}$ do not change if we replace $\mathcal{C}_n$ with $\mathcal{C}_n^{\epsilon}=\mathcal{C}_n \cap B_{\epsilon}(\lambda)$, and replace the contour $\mathcal{\gamma}_r$ defining $\mathsf{K}_n$ with the contour $\mathcal{\gamma}_r^{\epsilon}=\mathcal{\gamma}_r \cap B_{\epsilon}(0).$

First we change variables setting $z=\lambda+n^{-1/9} \overline{z},\o=\lambda+n^{-1/9} \overline{\o}$, and $\o'=\lambda+n^{-1/9} \overline{z}$. 
\begin{definition} Define the contours $\mathcal{D}_0=[-\i \infty, \i \infty]$, and $\mathcal{D}_0^{ \delta}=\mc{D}_0 \cap B_{ \delta}(0).$ (We will often use $\delta=n^{1/9} \epsilon$.)
\end{definition}
Our change of variables applied to the kernel $\msf{K}_n^{\epsilon}$ gives
\begin{multline}\overline{\mathsf{K}}^{\epsilon}_n(\overline{\o},\overline{\o}')=\frac{1}{2 \pi \i}\int_{\mathcal{D}_0^{n^{1/9} \epsilon}} \frac{1}{(\overline{z}-\overline{\o})(\overline{z}-\overline{\o}')} \frac{(\lambda+n^{-1/9} \overline{z})}{(\lambda+n^{-1/9}\overline{\o})} e^{n^{1/3} f_1(\lambda+n^{-1/9} \overline{z})-f_1(\lambda+n^{-1/9} \overline{\o})} \\ \times e^{n^{1/9} f_2(\lambda+n^{-1/9} \overline{z})-f_2(\lambda+n^{-1/9} \overline{\o})} e^{r_n(\lambda+n^{-1/9} \overline{z})-r_n(\lambda+n^{-1/9} \overline{\o})} d \overline{z}. \label{K def}
\end{multline}

\begin{definition} The contours $\mathcal{C}_{-1}$ and $\mathcal{C}_{-1}^{\epsilon}$ are defined as $\mathcal{C}_{-1}=(e^{-2\pi \i/3} \infty, -1) \cup [-1,e^{2 \pi \i/3} \infty)$ and $\mathcal{C}_{-1}^{\epsilon}=\mathcal{C}_{-1} \cap B_{n^{1/9} \epsilon}(-1).$ 
\end{definition}
By changing variables, for each $m$ we have 

$$\int_{(\mc{C}_{n}^{\epsilon})^m} \det(\mathsf{K}^{\epsilon}_n(\o_i,\o_j))_{i,j=1}^m d \o_1...d\o_m=\int_{(\mathcal{C}_{-1}^{n^{1/9}\epsilon})^m} \det(\overline{\mathsf{K}}_{n}^{\epsilon}(\overline{\o}_i,\overline{\o}_j))_{i,j=1}^m d \overline{\o}_1...d \overline{\o}_m.$$
This equality follows, because after rescaling the contour $C_{n}^{\epsilon}$, we can deform it to the contour $\mathcal{C}_{-1}^{n^{1/9} \epsilon}$ without changing its endpoints. The previous equality implies
$$\det(1-\mathsf{K}^{\epsilon}_n)_{L^2(\mathcal{C}_{\epsilon}^{\epsilon})}=\det(1-\overline{\mathsf{K}}_{n}^{ \epsilon})_{L^2(\mathcal{C}_{-1}^{n^{1/9} \epsilon})}.$$

We will make this change of variables often in the following arguments. Given a contour such as $\mc{C}_n$ or $\mc{\gamma}_r$, we denote the contour after the change of variables by $\mc{\overline{C}}_n$ or $\mc{\overline{\mathcal{\gamma}}}_r$. Now we are ready to localize our integrals.

\begin{proposition}
	\label{cut}
	For any sufficiently small $\epsilon>0$, 
	
	$$\lim_{n \to \infty} \det(1-\mathsf{K}_n(\o,\o'))_{L^2(\mathcal{C})}=\lim_{n \to \infty} \det(1-\mathsf{K}_{n}^{\epsilon}(\o,\o'))_{L^2(\mathcal{C}_{n}^{\epsilon})},$$
	where 
	
	$$\mathsf{K}_{n}^{\epsilon}=\frac{1}{2 \pi \i} \int_{\mathcal{\gamma}_r^{\epsilon}} \frac{e^{n^{1/3}t(z-\o)+h_n(z)-h_n(\o)}}{(z-\o)(z-\o')} \frac{z}{w} dz.$$
\end{proposition}

\begin{proof}
	The proof will have two steps, and will use several lemmas that are proved in Section 4. In the first step we localize the integral in the $z$ variable and show that $\lim_{n \to \infty}\det(1-\mathsf{K}_n)_{L^2(\mathcal{C}^{\epsilon})}=\lim_{n \to \infty}\det(1-\mathsf{K}_n^{\epsilon})_{L^2(\mathcal{C}^{\epsilon})}$ using dominated convergence. In order to prove this, we appeal to Lemmas \ref{dom} and \ref{dom 2} to show that the Fredholm series expansions are indeed dominated. In the second step we localize the integral in the $\o,\o'$ variables by using Lemma \ref{max bound lemma} to find an upper bound for $\det(1+K_n)_{L^2(\mc{C}_n)}-\det(1+K_n)_{L^2(\mc{C}_n^{\epsilon})}.$ Then we appeal to Lemma \ref{sum} to show that this upper bound converges to $0$ as $n \to \infty$. 
	
	\item[\textbf{Step 1:}] \quad By Lemma \ref{gamma bound}, for any $\epsilon>0$, there exists a $C',N>0$ such that if $\o \in C_n$ and $z \in \mathcal{\gamma}_r \setminus \mathcal{\gamma}_r^{\epsilon}$, then for all $n>N$,
	
	$$\Re[h_n(z)-h_n(\o)+n^{1/3}t(z-\o)] \leq -n^{1/3} C'.$$
	
	We bound our integrand on $\mathcal{\gamma}_r \setminus \mathcal{\gamma}_r^{\epsilon}$, $ \o,\o' \in \mc{C}_n^{\epsilon},$ 
	$$\left|\frac{e^{h_n(z)-h_n(\o)+n^{1/3}t(z-\o)}}{(z-\o)(z-\o')} \frac{z}{\o}\right| \leq \frac{C}{\delta^2} z e^{-n^{1/3 C'}} \xrightarrow[n \to \infty]{pointwise} 0.$$ 
	(the $\delta^2$ comes from the fact that $|z-\o| \geq \delta$). By Lemma \ref{chaos control}, there exists a $\eta>0$ such that for sufficiently large $n$,
	$$ \left|\frac{e^{h_n(z)-h_n(\o)+n^{1/3}t(z-\o)}}{(z-\o)(z-\o')} \frac{z}{\o}\right| <\left|\frac{e^{n^{1/3}(f_1(z)-f_1(\o)+\eta)}}{(z-\o)(z-\o')} \frac{z}{\o}\right|.$$
	The linear term of $f_1(z)$ in (\ref{f def}) implies
	$$\frac{1}{2 \pi \i}\int_{\mc{\gamma}_r} \left|\frac{e^{n^{1/3}(f_1(z)-f_1(\o)+\eta)}}{(z-\o)(z-\o')} \frac{z}{\o}\right| dz<\infty.$$
	\medskip
	
	In the previous inequality we should write $|dz|$ instead of $dz$. We will often omit the absolute value in the $d\o$ portion of the complex integral when the integrand is a positive real valued function. 
	
	\medskip

	\noindent So for each $\o,\o'$, by dominated convergence 
	$$\frac{1}{2 \pi \i}\int_{\mathcal{\gamma}_r \setminus \mathcal{\gamma}_r^{\epsilon}}\frac{e^{h_n(z)-h_n(\o)+n^{1/3}t(z-\o)}}{(z-\o)(z-\o')} \frac{z}{\o} dz \to 0 \quad \text{as} \quad n \to \infty,$$
	So $\lim_{n \to \infty}\mathsf{K}_n^{\epsilon}(\o,\o') = \lim_{n \to \infty} \mathsf{K}_n(\o,\o').$
	
	Now by Lemma \ref{dom}, and \ref{dom 2}, both Fredholm determinant expansions $\det(1-\mathsf{K}_n)_{L^2(\mathcal{C}^{\epsilon})}$ and $\det(1-\mathsf{K}^{\epsilon}_n)_{L^2(\mathcal{C}^{\epsilon})}$, are absolutely bounded uniformly in $n$. Thus we can apply dominated convergence to get
	
	\begin{equation}\lim_{n \to \infty}\det(1-\mathsf{K}_n)_{L^2(\mathcal{C}^{\epsilon})}=\lim_{n \to \infty}\det(1-\mathsf{K}_n^{\epsilon})_{L^2(\mathcal{C}^{\epsilon})}. \label{fred equal} \end{equation}
	
	\item[\textbf{Step 2:}] \quad In the expansion 
	$$\det(1-\mathsf{K}_n)_{L^2(\mathcal{C}_n)}=\sum_{m=0}^{\infty} \frac{1}{m!} \int_{(\mc{C}_n)^m} \det(\mathsf{K}_n(\o_i,\o'_j))_{i,j=1}^n d\o_1,...,d\o_m.$$ 
	The $m$th term can be decomposed as the sum
	$$\int_{(\mc{C}_n^{\epsilon})^m} \det( \mathsf{K}_n(\o_i,\o_j))_{i,j=1}^n d\o_1...d\o_m+\int_{\mc{C}_n^m \setminus (\mc{C}_n^{\epsilon})^m} \det(\mathsf{K}_n(\o_i,\o_j))_{i,j=1}^n d\o_1...d\o_m.$$
	Lemma \ref{max bound lemma} along with Hadamard's bound on the determinant of a matrix in terms of it's row norms, implies that when $\o_1\in \mc{C}_n \setminus \mc{C}_n^{\epsilon}$ and $\o_2,...,\o_m \in \mc{C}^n$,
	\begin{equation} |\det(\overline{\mathsf{K}}_{n}(\o_i, \o_j))_{i,j=1}^m| \leq  m^{m/2} M^{m-1/2} L_4 n^{4/9} e^{-n^{1/3} \eta}  \to 0 \text{ as } n \to \infty. \label{determinant bound}\end{equation}
	
	Now let $R$ be the maximum length of the paths $\mc{C}_n$. The rescaled paths $\overline{\mc{C}_n}$ will always have length less than $n^{1/9}R$. We have
	
	\begin{align}
	&\int_{\mc{C}_n^{m} \setminus (\mc{C}_n^\epsilon)^m} |\det(\mathsf{K}_n(\o_i, \o_j))_{i,j=1}^m| d\o_1...d\o_m\nonumber\\  
	&\hspace{2cm}\leq 
	m\int_{\mc{C}_n \setminus \mc{C}_n^{\epsilon}} d\o_1 \int_{\mc{C}_n^{m-1}} |\det(\mathsf{K}_n(\o_i, \o_j))_{i,j=1}^m| d\o_2...d\o_m \nonumber\\
	&\hspace{2cm}\leq m\int_{\overline{\mc{C}}_n \setminus \overline{\mc{C}}_n^{\epsilon}} d\overline{\o}_1 \int_{\overline{\mc{C}}_n^{m-1}}   |\det(\overline{\mathsf{K}}_{n}(\overline{\o}_i, \overline{\o}_j))_{i,j=1}^m|     d\overline{\o}_2...d\overline{\o}_m \nonumber\\
	&\hspace{2cm}\leq \int_{\overline{\mc{C}}_n \setminus \overline{\mc{C}}_n^{\epsilon}} d\overline{\o}_1 \int_{\overline{\mc{C}}_n^{m-1}} m^{m/2} M^{(m-1)/2} L_4 n^{4/9} e^{-n^{1/3} \eta} d\overline{\o}_2...d\overline{\o}_m \nonumber \\
	&\hspace{2cm}\leq m (n^{1/9}R)^{m} m^{m/2} M^{(m-1)/2} L_4 n^{4/9} e^{-n^{1/3} \eta}\nonumber\\
	&\hspace{2cm}\leq  e^{-n^{1/3} \eta} (n^{1/9})^mm^{1+m/2} (MR)^m n^{4/9}.
	\end{align}
	
	The first inequality follows from symmetry of the integrand in the $\o_i$. In the second inequality, we change variables from $\o_i$ to $\overline{\o}_i$. In the third inequality we use the first inequality of (\ref{determinant bound}). In the fourth inequality, we use the fact that the total volume of our multiple integral is less than $(n^{1/9}R)^m$. In the fifth inequality we rewrite and use $M^m>M^{(m-1)/2}$. 
	
	So we have
	
	\begin{multline}
	\sum_{m=1}^{\infty} \frac{1}{m!} \int_{\mc{C}_n^m \setminus (\mc{C}_n^{\epsilon})^m} |\det( \mathsf{K}_n(\o_i,\o_j))_{i,j=1}^m| d\o_1...d\o_m \\  \leq 
	\sum_{m=1}^{\infty} \frac{1}{m!}e^{-n^{1/3} \eta} (n^{1/9})^mm^{1+m/2} (MR)^m n^{4/9}  \\
	\\ = n^{4/9} e^{-n^{1/3} \eta} \sum_{m=1}^{\infty} \frac{1}{m!} (MRn^{1/9})^m m^{1+m/2} \label{sum bound}
	\end{multline}
	
	Applying Lemma \ref{sum} with $C=MRn^{1/9}$ gives. 
	
	$$n^{4/9} e^{-n^{1/3} \eta} \sum_{m=1}^{\infty} \frac{1}{m!} (MRn^{1/9})^m m^{1+m/2} \leq  n^{4/9}e^{-n^{1/3}} 16 (MRn^{1/9})^4 e^{2 (MR)^2 n^{2/9}} \xrightarrow[n \to \infty]{} 0.$$
	Thus 
	\begin{equation}\lim_{n \to \infty}\det(1-\mathsf{K}_n)_{L^2(\mc{C}_n)}=\lim_{n \to \infty} \det(1-\mathsf{K}_n)_{L^2(\mc{C}_n^{\epsilon})}. \label{fred equal 2} \end{equation}
	Combining (\ref{fred equal}) and (\ref{fred equal 2}) concludes the proof of Proposition 2.11.
	
\end{proof}

\subsection{Convergence of the kernel}

In this section we approximate $h_n(z)-h_n(\o)+n^{1/3}t(z-\o)$ by its Taylor expansion near $\lambda$, and show that this does not change the asymptotics of our Fredholm determinant. 

\begin{proposition} \label{taylor approximation}
	For sufficiently small $\epsilon>0$,

	$$\lim_{n \to \infty} \det(1-\mathsf{K}_n^{\epsilon})_{L^2(\mathcal{C}_{\epsilon}^{\epsilon})}=\lim_{n \to \infty}\det(1-\mathsf{K}_{(x)})_{L^2(\mathcal{C}_{-1})},$$
	where 
	$$\mathsf{K}_{(x)}(\overline{u},\overline{u}')=\frac{1}{2\pi \i}\int_{D'} \frac{e^{s^3/3-xs}}{e^{u^3-xu}} \frac{dz}{(z-u)(z-u')},$$
	and
	$$D'=(e^{-\pi \i /3} \infty, 0) \cup [0,e^{\pi \i/3} \infty).$$
\end{proposition}

\begin{proof} Let \begin{equation}\mathsf{K}(\overline{\omega},\overline{\omega}')=\frac{1}{2 \pi \i} \int_{D'} \frac{d\overline{z}}{(\overline{z}-\overline{\o})(\overline{z}-\overline{\o}')} e^{ f_1'''(\lambda) (\overline{z}^3-\overline{\o}^3)/6+f_2'(\lambda)(\overline{z}-\overline{\o})}, \label{def:K} \end{equation}
	We have seen in Section \ref{localize} that
	
	$$\det(1-\mathsf{K}_n^{\epsilon}(\o,\o'))_{L^{2}(\mathcal{C}_{\epsilon}^{\epsilon})}=\det(1-\overline{\mathsf{K}}_{n}^{\epsilon}(\overline{\o},\overline{\o}'))_{L^{2}(\mathcal{C}_{-1}^{n^{1/9}\epsilon})}.$$
	
	The proof will have two main steps. In the first step we use dominated convergence to show that 
	$$\lim_{n \to \infty} \det(1-\overline{\mathsf{K}}_{n}^{\epsilon}(\overline{\o},\overline{\o}'))_{L^{2}(\mathcal{C}_{-1}^{n^{1/9}\epsilon})} = \lim_{n \to \infty} \det(1-\overline{\mathsf{K}}_{(x)}(\overline{\o},\overline{\o}'))_{L^{2}(\mathcal{C}_{-1}^{n^{1/9}\epsilon})}.$$
	In the second step we control the tail of the Fredholm determinant expansion to show that
	$$\lim_{n \to \infty} \det(1-\overline{\mathsf{K}}_{(x)}(\overline{\o},\overline{\o}'))_{L^{2}(\mathcal{C}_{-1}^{n^{1/9}\epsilon})}=\det(1-\overline{\mathsf{K}}_{(x)}(\overline{\o},\overline{\o}'))_{L^{2}(\mathcal{C}_{-1})}.$$
	In step $1$ we will use Lemma \ref{dom} to establish dominated convergence. 
	
	\item[\textbf{Step 1:}]We have the following pointwise convengences
	
	$$\frac{\lambda+n^{-1/9} \overline{z}}{\lambda + n^{-1/9} \overline{\o}} \to 1,$$
	and for $z=\lambda+n^{-1/9} \bar{z}, \o=\lambda+n^{-1/9}\overline{\o}$, 
	\begin{equation} n^{1/3}( f_1(z)-f_1(\o))+n^{1/9} (f_2(z)-f_2(\o))+r_n(z)-r_n(\o)  \rightarrow \frac{1}{6}f_1'''(\lambda)(\overline{z}^3-\overline{\o}^3)+f'_2(\lambda)(\overline{z}-\overline{\o}). \label{exponent} \end{equation} 
	Because $z$ is purely imaginary, for each $\overline{\o}, \overline{\o}'$, the exponentiating the right hand side of (\ref{exponent}) gives a bounded function of $\overline{z}$ and $z/\o \leq \frac{|\lambda+\epsilon|}{|\lambda-\epsilon|}$. The left hand side of (\ref{exponent}) can be chosen to be within $\delta/n^{1/9}$ of the right hand side by choosing $\epsilon$ small by Taylor's theorem, because all the functions on the left hand side are holomorphic in $B_{\epsilon}(\lambda)$. Thanks to the quadratic denominator $\frac{1}{(\overline{z}-\overline{\o})(\overline{z}-\overline{\o}')}$, we can apply dominated convergence to get
	
	\begin{equation}\overline{\mathsf{K}}_{n}^{\epsilon}(\overline{\o},\overline{\o}') \xrightarrow[n \to \infty]{pointwise} \frac{1}{2\pi \i} \int_{\i \R} \frac{d\overline{z}}{(\overline{z}-\overline{\o})(\overline{z}-\overline{\o}')} e^{ f_1'''(\lambda)  (\overline{z}^3-\overline{\o}^3)/6+f_2'(\lambda)(\overline{z}-\overline{\o})}. \label{conv} \end{equation}
	Because the integrand on the right hand side of (\ref{conv}) has quadratic decay in $\overline{z}$, we can deform the contour from $\mathcal{\gamma}_0$ to $D'$ without changing the integral, so the right hand side is equal to $\mathsf{K}(\overline{\o},\overline{\o}')$ from \ref{def:K}. Now by Lemma \ref{dom} we can apply dominated convergence to the expansion of the Fredholm determinant $\det(1-\overline{\mathsf{K}}_{n}^{\epsilon})_{L^2(\mathcal{C}_{-1}^{n^{1/9}\epsilon})}$, to get
	
	$$\lim_{n \to \infty}\det(1-\overline{\mathsf{K}}_{n}^{\epsilon})_{L^2(\mathcal{C}_{-1}^{n^{1/9}\epsilon})}=\lim_{n \to \infty}\det(1-\mathsf{K})_{L^2(\mathcal{C}_{-1}^{n^{1/9} \epsilon})}.$$
	\item[\textbf{Step 2:}] Now we make the change of variables $s=-(f_2'(\lambda)/x) \overline{z}$, $u=-(f_2'(\lambda)/x) \overline{\o}$, and $u'=-(f_2'(\lambda)/x) \overline{\o}'$. Keeping in mind that $-2(f_2'(\lambda)/x)^3=f_1'''(\lambda)$, we get
	
	$$\mathsf{K}(\overline{\o},\overline{\o}')=\mathsf{K}_{(x)}(u,u')=\frac{1}{2\pi \i}\int_{D'} \frac{e^{s^3/3-xs}}{e^{u^3/3-xu}} \frac{ds}{(s-u)(s-u')}.$$
	Recall the expansion:
	$$\det(1-\mathsf{K}_{(x)})_{L^2(\mathcal{C}_{-1}^{\epsilon})}=\sum_{m=0}^{\infty} \frac{(-1)^m}{m!} \int_{\mathcal{C}_{-1}^m} \det(\mathsf{K}_{(x)}(\o_i,\o_j))_{i,j=1}^m d\o_1...d\o_m,$$
	
	where $\mathcal{C}_{-1}=(e^{-2\pi \i/3} \infty, 1] \cup (1, e^{2\pi \i/3} \infty)$, and $\mathcal{C}_{-1}^m$ is a product of $m$ copies of $\mathcal{C}_{-1}.$
	
	\begin{multline*}
	| \det(1-\mathsf{K}_{(x)})_{L^2(\mathcal{C}_{-1})}-\det(1-\mathsf{K}_{(x)})_{L^2(\mathcal{C}_{-1}^{\epsilon})} | \leq \\  \sum_{m=0}^{\infty} \frac{(-1)^m}{m!} \int_{\mathcal{C}_{-1}^m \setminus (\mathcal{C}_{-1}^{n^{1/9} \epsilon })^m} |\det (\mathsf{K}_{(x)}(\o_i,\o_j))_{i,j=1}^m| d\o_1...d\o_m,
	\end{multline*}
	
	so to conclude the proof of the proposition, we are left with showing that
	
	\begin{equation} \sum_{m=0}^{\infty} \frac{1}{m!} \int_{\mathcal{C}_{-1}^m \setminus (\mathcal{C}_{-1}^{n^{1/9} \epsilon })^m} |\det (\mathsf{K}_{(x)}(\o_i,\o_j))_{i,j=1}^m| d\o_1...d\o_m \xrightarrow[n \to \infty]{} 0 \label{step 2} \end{equation}
	
	Note that 
	\begin{multline*}
\int_{\mathcal{C}_{-1}^m \setminus (\mathcal{C}_{-1}^{ n^{1/9}\epsilon})^m} |\det (\mathsf{K}_{(x)}(\o_i,\o_j))_{i,j=1}^m| d\o_1...d\o_m \leq \\ m \int_{\mathcal{C}_{-1} \setminus \mathcal{C}_{-1}^{n^{1/9}\epsilon}} \int_{\mathcal{C}_{-1}^{m-1}} |\det (\mathsf{K}_{(x)}(\o_i,\o_j))_{i,j=1}^m| d\o_1...d\o_m.
	\end{multline*}
	Set
	$$M_1=\int_{D'}|\overline{z}e^{ f_1'''(\lambda) \overline{z}^3/6+f_2'(\lambda)\overline{z}}|d\overline{z}<\infty.$$
	
	Then $\mathsf{K}_{(x)}(\o,\o') \leq M_1 e^{-|\o|^3-x |\omega|}$, and Hadamard's bound gives 
	$$|\det (\mathsf{K}_{(x)}(\o_i,\o_j))_{i,j=1}^m| \leq m^{m/2}M_1^m \prod_{i=1}^m |e^{-\o_i^3/3+x \omega_i}|.$$ 
	We have
	
	\begin{align}
	&\int_{\mathcal{C}_{-1} \setminus \mathcal{C}_{-1}^{ n^{1/9}\epsilon }} \int_{\mc{C}_{-1}^{m-1}} |\det (\mathsf{K}_{(x)}(\o_i,\o_j))_{i,j=1}^m| d\o_1...d\o_m\nonumber\\
	 &\hspace{2cm}\leq M_1\int_{\mathcal{C}_{-1} \setminus \mathcal{C}_{-1}^{ n^{1/9}\epsilon }} \int_{\mc{C}_{-1}^{m-1}} \prod_{i=1}^m |e^{-\o_i^3/3+x \omega_i}| d\o_1...d\o_m\nonumber \\
	&\hspace{2cm}\leq  m^{1+m/2} M_1^m M_2^{m-1} \int_{\mathcal{C}_{-1} \setminus \mathcal{C}_{-1}^{n^{1/9}\epsilon}} | e^{-\o_1^3+x \omega_1}| d\omega_1, \label{to 0}
	\end{align}
	where $M_2=\int_{\mathcal{C}_{-1}} |e^{-\o^3-x \omega}| d \omega <\infty$ because $-\o^3$ lies on the negative real axis. (\ref{to 0}) goes to zero because $n^{1/9} \epsilon \to \infty$. So
	$$\int_{\mathcal{C}_{-1} \setminus \mathcal{C}_{-1}^{ n^{1/9}\epsilon }} \int_{\mathcal{C}_{-1}^{m-1}} \left|\det (\mathsf{K}_{(x)}(\o_i,\o_j))_{i,j=1}^m\right| d\o_1...d\o_m \xrightarrow[n \to \infty]{} 0.$$
	Note also that
	
	\begin{align*} 
	\int_{\mathcal{C}_{-1}^{m}\setminus (\mathcal{C}_{-1}^{ n^{1/9}\epsilon})^m} \left|\det (\mathsf{K}_{(x)}(\o_i,\o_j))_{i,j=1}^m \right| d\o_1...d\o_m &\leq 
	\int_{\mathcal{C}_{-1}^{m}} |\det (\mathsf{K}_{(x)}(\o_i,\o_j))_{i,j=1}^m | d\o_1...d\o_m \\  &\leq m^{1+m/2} M_1 M_2^m.
	\end{align*}
	By Stirling's approximation
	$$\sum_{m=0}^{\infty} \frac{1}{m!} m^{1+m/2} M_1^m M_2^m< \infty.$$
	So by dominated convergence (\ref{step 2}) holds which concludes the proof of Proposition \ref{taylor approximation}.
\end{proof}

\subsection{Reformulation of the kernel}

Now we use the standard $\det(1+AB)=\det(1+BA)$ trick \cite[Lemma 8.6]{FreeEnergyCorwin} to identify $\det(1-\mathsf{K}_{(x)})_{L^2(\mathcal{C}_{-1})}$ with the Tracy-Widom cumulative distribution function.
\begin{lemma} For $x \in \mathbb{R}$,
	$$\det(1-\mathsf{K}_{(x)})_{L^2(\mathcal{C}_{-1})} =\det(1-\mathsf{K}_{\mathrm{Ai}})_{L^2(x,\infty)}.$$
\end{lemma}

\begin{proof}
	First note that because $\Re[z-\o] > 0$ along the contours we have chosen, we can write 
	$$\frac{1}{z-\o}=\int_{\R_+} e^{-\lambda(z-\o)} d\lambda.$$
	Now let $A:L^2(\mathcal{C}_{-1}) \to L^2(\R_+)$, and $B:L^2(\R_+) \to L^2(\mathcal{C}_{-1})$ be defined by the kernels
	\begin{align}
	A(\omega,\lambda)&=e^{-\o^3/3+\o(x+\lambda)},\\
	B(\lambda,\omega')&=\int_{e^{-\pi \i/3} \infty}^{e^{\pi \i/3} \infty} \frac{dz}{2\pi \i} \frac{e^{z^3/3-z(x+\lambda)}}{z-\o'}.
	\end{align}
	We compute
	\begin{align*}
	AB(\o,\o') &=\int_{\R_+} e^{-\o^3/3+\o(x+\lambda)} \int_{e^{-\pi \i/3}\infty}^{e^{\pi \i/3} \infty} \frac{dz}{2 \pi \i} \frac{e^{z^3/3-z(x+\lambda)}}{z-\o'}\\ 
	&= \frac{1}{2\pi \i} \int_{e^{-\pi \i/3} \infty}^{e^{\pi \i/3} \infty} \frac{e^{z^3/3-zx}}{e^{\o^3/3-\o x}} \frac{dz}{(z-\o)(z-\o')}\\
	&=\mathsf{K}_{(x)}(\o,\o').
	\end{align*}
	%
	%
	Similarly,
	$$BA(s,s')=\frac{1}{2\pi \i} \int_{e^{-2 \pi \i/3} \infty}^{e^{2 \pi \i/3} \infty} d\o \frac{1}{2\pi \i} \int_{e^{-\pi \i/3} \infty}^{e^{\pi \i/3} \infty} dz \frac{e^{z^3/3-z(x+s)}}{e^{\o^3/3-\o(x+s')}} \frac{1}{(z-\o)}
	=\mathsf{K}_{\mrm{Ai}}(x+s,x+s').$$
	Because both $A$ and $B$ are Hilbert-Schmidt operators, we have
\begin{multline*}
\det(1-\mathsf{K}_{(x)})_{L^2(\mathcal{C})}=\det(1-AB)_{L^2(\R_+)} = \det(1-BA)_{L^2(\R_+)}\\ =\det(1-\mathsf{K}_{\mrm{Ai}})_{L^2(x,\infty)}=F_{\textrm{GUE}}(x).
\end{multline*}
\end{proof}

\section{Constructing the contour $\mathcal{C}_n$}
\label{Cn}

This section is devoted to constructing the contours $\mathcal{C}_n$ and proving Lemma \ref{C bound}. We will prove several estimates for $n^{1/3}\o+h_n(\o)$; then we will construct the contour $\mc{C}_n$, and prove it satisfies the properties of Lemma \ref{C bound}. We begin by proving that we can approximate $n^{1/3}\o+h_n(\o)$ by $n^{1/3}f_1(\o)$ away from $0$.

\subsection{Estimates away from 0: proof of Lemma \ref{chaos control}}


Both inequalities for $|f_2|=\frac{b \sigma x}{\o}$ follow from the fact that $f_2$ and $f_2'$ are bounded on $\C \setminus B_{\epsilon}(0)$. Let $y=1/\o$, and let $m=n^{-1/9}$. Define the function $g(y,m)=r_n(\o).$ First we prove (\ref{control}). Note that $h_n(\o)$ is holomorphic in $y$ and $m$ except when $n =\infty$, $n^{1/3} \omega=0,-a-b$. By Taylor expanding $h_n(\o)$, we see that $r_n(\o)=g(y,m)$ is holomorphic in $y$ and $m$, except at points $(y,m)$ such that $n^{1/3} \omega=0,-a-b$, in particular there is no longer a pole when $n=\infty$. Thus for any $N$, $g(y,m)$ is holomorphic with variables $y$ and $m$, in the region $U=\{(y,m):n>N, \o> |a+b|/N^{1/3}\}$, because in this region $n^{1/3} \omega>|a+b|$. 
The region $U_{\epsilon}=\{(y,m): n>N, \o \geq \frac{|a+b|+\epsilon}{N^{1/3}} \}$ is compact in the variables $y$ and $m$, and because $U_{\epsilon} \subset U$, the function $g(y,m)$ is holomorphic in the region $U_{\epsilon}$. Thus $g(y,m)=r_n(\o)$ is bounded by a constant $C$ in the region $U_{\epsilon}$.

Now we prove (\ref{derivative control}). For any $\delta$, pick an arbitrary $\epsilon$ and an $N_{\delta}$ large enough that $\frac{|a+b|+\epsilon}{N_{\delta}^{1/3}} \leq \delta$. Because $g(y,m)=r_n(\o)$ is holomorphic in the variables $y$ and $m$ in the compact set $U_{\epsilon}$, the function $\frac{\partial}{\partial y} g(y,m)=-\o^2 r_n'(\o)$, is also holomorphic in $y,m$. So $|\o^2 r_n'(\o)| \leq C$ on $U_{\epsilon}$. We rewrite as $|r_n'(\o)| \leq C/|\o|^2$, and this gives $|r_n'(\o)| \leq \frac{C}{|\delta|^2} \leq C',$ on the set  $U_{\epsilon} \cap (\mathbb{N} \times B_{\delta}(0)^c)$. But by our choice of $N_{\delta}$, we have $U_{\epsilon} \cap (\mathbb{N} \times B_{\delta}(0)^c)$ is just the set $\{(y,m): n \geq N_{\delta}, |\o| \geq \delta \}$. 


\subsection{Estimates near 0}

The function $n^{1/3}f_1(\o)$ only approximates $-n^{1/3}t \o -h_n(\o)$ well away from $0$.
In this section we give two estimates for $-n^{1/3}t \o -h_n(\o)$: one in Lemma \ref{one third bound} when $\o$ is of order $n^{-1/3}$ and one in Lemma \ref{delta bound} when $\o$ is of order $n^{\delta-1/3}$ for $\delta \in (0,1/3).$ Together with Lemma \ref{chaos control} which gives an estimate when $\o$ is of order $1$, this will give us the tools we need to control $-n^{1/3}t \o -h_n(\o)$ along $\mathcal{C}_n$. First to prove the bound in Lemma \ref{one third bound}, we choose a path which crosses the real axis at $-a$, between the poles at $0$ and $-a-b$ before rescaling $\tilde{h}_n$ to $h_n$. We show that after the rescaling, we can bound $\Re[-n^{-1/3} \o-h_n(\o)]$ on this path for small $\o$.

\begin{lemma} 
	\label{one third bound} Fix any $c_0>1$ and let $s=c_0(a+b)$. For $C=\log\left(\sqrt{s^2+a^2}\right)-\log(s)>0$, we have
	$$\limsup_{n \to \infty} \frac{1}{n} \sup_{y \in [-s,s]} \Re[h_n(\lambda)-h_n(\mbf{i}n^{-1/3}y-n^{-1/3}a)]< -C.$$
\end{lemma}

\begin{proof}
	Let $y \in [-s,s]$ and expand $e^{\Re[h_n(\lambda)-h_n(iy-a n^{-1/3})]}$ to get
	
	$$\left(\frac{y}{\sqrt{y^2+a^2}}\right)^n \left(\frac{y}{\sqrt{y^2+b^2}} \right)^m \left(\frac{n^{1/3} \lambda}{n^{1/3} \lambda+a} \right)^n \left(\frac{a+b+n^{1/3} \lambda}{n^{1/3} \lambda+a} \right)^m.$$
	The third factor is always less than $1$. For sufficiently large $n$, the second factor times the fourth factor is less than $1$, because $|y| \leq |s|$ while $n^{1/3} \lambda \to \infty$. We can bound the first factor by
	$$\left|\frac{y}{\sqrt{y^2+a^2}}\right|^n \leq \left(\frac{s}{\sqrt{s^2+a^2}}\right)^n = e^{-nC},$$
	with $C=\log\left(\sqrt{(s^2+a^2)}\right)-\log(s)$.
\end{proof}

Next we will prove the estimate for $\o$ of order $n^{\delta-1/3}$. In this proof we will consider $\o$ of the form $\o=-n^{-1/3}a+\i n^{\delta-1/3}c(a+b)$, choose $c$ sufficiently large, then let $n \to \infty$. The largest term in the expansion of $-n^{-1/3}\o-h_n(\o)$ will be of order $\frac{n^{1-2\delta}}{c^2}$. We introduce the following definition to let us ignore the terms which are negligible compared to $\frac{n^{1-2\delta}}{c^2}$ uniformly in $\delta.$

\begin{definition} Let $A$ and $B$ be functions depending on $n$ and $c$, we say $A \sim_{\delta} B$ or $A$ is $\delta$-equivalent to $B$, if for sufficiently large $c$ and $n$,
	$$|A-B| \leq \frac{n^{2/3-2\delta}}{c^2}M_1+\frac{n^{1-3\delta}}{c^3}M_2+\frac{n^{4/9-\delta}}{c}M_3.$$ for some constants $M_1,M_2,M_3$ independent of $c$ and $n$. 
\end{definition}

Now we prove the estimate. 

\begin{lemma}
	For all $\delta \in (0,1/3)$, setting $\o=-n^{-1/3}a+\i n^{\delta-1/3}c(a+b)$, gives 
	$$\Re[n^{1/3}t \o+h_n(\o)] \sim_{\delta} \Re[n^{1/3}f_1(\o)] \sim_{\delta} M\frac{n^{1-2\delta}}{c^2},$$
	where $\sim_{\delta}$ is defined in Definition 8.
	\label{delta bound} 
\end{lemma}

The proof of this Lemma \ref{delta bound} comes from Taylor expanding $h_n$ and keeping track of the order of different terms with respect to $n$ and $c$.
\begin{proof} Recall that
	\begin{equation}h_n(\o)=-n \log \left(1+\frac{a}{n^{1/3} \o}\right)+m \log\left(1+\frac{b}{a+n^{1/3}\o}\right). \label{hn} \end{equation}
	
	For $|n^{1/3} \o|>a$ and $|a+n^{1/3}\o|>b$, we can Taylor expand in $n^{1/3} \o$ to get 
	
	$$h_n(\o)=-n \sum_{k=1}^{\infty} \frac{(-1)^{k+1}}{k}\left(\frac{a}{n^{1/3} \o}\right)^k+m \sum_{k=1}^{\infty} \frac{(-1)^{k+1}}{k}\left(\frac{b}{a+n^{1/3}\o}\right)^k.$$
	
	Let $\o=-n^{-1/3}a+\i n^{\delta-1/3}c(a+b)$ for $\delta \in (0,1/3)$, so $ |n^{1/3} \o|, |a+n^{1/3}\o| > n^{\delta}c(a+b)>c(a+b)$, for a constant $c$ to be determined later. If $c>2$, we have
	
	\begin{equation}\sum_{k=1}^{\infty} \left|\left(\frac{a}{n^{1/3} \o}\right)\right|^k \leq \sum_{k=1}^{\infty} \left(\frac{b}{n^{\delta}c(a+b)}\right)^k \leq \frac{a}{n^{\delta}c(a+b)}\sum_{k=0}^{\infty}\left(\frac{1}{2}\right)^k \leq \frac{2a}{n^{\delta}c(a+b)}=\frac{n^{-\delta}}{c}M, \label{sum 1}\end{equation}
	
	and
	
	\begin{equation} \sum_{k=1}^{\infty} \left|\left(\frac{b}{a+n^{1/3}\o}\right)\right|^k \leq \sum_{k=1}^{\infty} \left(\frac{a}{n^{\delta}c(a+b)}\right)^k \leq \frac{a}{n^{\delta}c(a+b)}\sum_{k=0}^{\infty} \left(\frac{1}{2}\right)^k=\frac{2a}{n^{\delta}c(a+b)}=\frac{n^{-\delta}}{c} M. \label{sum 2} \end{equation}
	
	In what follows, we will use (\ref{sum 1}) or (\ref{sum 2}) when we say that an infinite sum is $\delta$-equivalent to its first term.
	
	We examine the first term in (\ref{hn}).
	\begin{align*}
	-n\sum_{k=1}^{\infty} \frac{(-1)^{k+1}}{k}\left(\frac{a}{n^{1/3} \o}\right)^k &= -\left(\frac{a}{n^{1/3} \o}\right)+\frac{1}{2}\left(\frac{a}{n^{1/3} \o}\right)^2-n\sum_{k=3}^{\infty} \frac{(-1)^{k+1}}{k}\left(\frac{a}{n^{1/3} \o}\right)^k,\\
	&\sim_{\delta} -\left(\frac{a}{n^{1/3} \o}\right)+\frac{1}{2} \left(\frac{a}{n^{1/3} \o}\right)^2.
	\end{align*}
	where the $\delta-$equivalence follows because $\left|n\sum_{k=3}^{\infty} \frac{(-1)^{k+1}}{k} \left(\frac{a}{n^{1/3} \o}\right)^k\right|  \leq \frac{n^{1-3\delta}}{c^3}M$ for some $M$ by (\ref{sum 1}). 
	
	Recall that 
	$$m \sum_{k=1}^{\infty} \left(\frac{b}{a+n^{1/3}\o}\right)^k=\left[\left(\frac{a}{b}\right)n+dn^{2/3}+\sigma x n^{4/9}\right]\sum_{k=1}^{\infty}\left(\frac{b}{a+n^{1/3}\o}\right)^k.$$
	We decompose this series as three sums. First the $\left(\frac{a}{b}\right)n$ term gives
	\begin{multline*}
	\frac{a}{b} n \sum_{k=1}^{\infty} \frac{(-1)^{k+1}}{k}\left(\frac{b}{a +n^{1/3}\o}\right)^k= \\  n\left(\frac{a}{b}\right) \left(\frac{b}{a +n^{1/3}\o}\right)-\frac{n}{2}\left(\frac{a}{b}\right)\left(\frac{b}{a +n^{1/3}\o}\right)^2 +\frac{a}{b} n \sum_{k=3}^{\infty} \frac{(-1)^{k+1}}{k}\left(\frac{b}{a +n^{1/3}\o}\right)^k 
	\\ \sim_{\delta} n\left(\frac{a}{b}\right)\left(\frac{b}{a +n^{1/3}\o}\right)-\frac{n}{2}\left(\frac{b}{a +n^{1/3}\o}\right)^2,
	\end{multline*}
	because $\left|-\frac{a}{b} n \sum_{k=1}^{\infty} \frac{(-1)^{k+1}}{k}\left(\frac{b}{a +n^{1/3}\o}\right)^k\right| \leq  Mn^{1-3\delta}/c^3$ for some $M$.  The second term is 
	\begin{align*}
	dn^{2/3} \sum_{k=1}^{\infty} \frac{(-1)^{k+1}}{k}\left(\frac{b}{a+n^{1/3} \o}\right)^k &= dn^{2/3}\left(\frac{b}{a+n^{1/3} \o}\right)-dn^{2/3} \sum_{k=2}^{\infty} \frac{(-1)^{k+1}}{k}\left(\frac{b}{a+n^{1/3} \o}\right)^k\\
	&\sim_{\delta} dn^{2/3}\left(\frac{b}{a+n^{1/3} \o}\right)
	\end{align*}
	because $\left|dn^{2/3} \sum_{k=2}^{\infty} \frac{(-1)^{k+1}}{k}\left(\frac{b}{a+n^{1/3} \o}\right)^k\right| \leq  Mn^{2/3-2\delta}/c^2$ for some $M$.  The third term is 
	$$n^{4/9} \sigma x \sum_{k=1}^{\infty} \frac{(-1)^{k+1}}{k}\left(\frac{b}{a+n^{1/3} \o}\right)^k \sim_{\delta} 0,$$
	because the full sum $\left|n^{4/9} \sigma x \sum_{k=1}^{\infty} \frac{(-1)^{k+1}}{k}\left(\frac{b}{a+n^{1/3} \o}\right)^k\right| \leq \frac{Mn^{4/9-\delta}}{c}$ for some $M$. 
	Now we have shown 
	\begin{equation}
	-n \log\left(1+\frac{a}{n^{1/3} \o}\right) \sim_{\delta} -n^{2/3} \frac{a}{\o}+n^{1/3} \frac{a^2}{2\o^2}, \label{term 1}
	\end{equation}
	\begin{multline}
	m \log\left(1+\frac{b}{a+n^{1/3}\o}\right) \sim_{\delta}   \\ n \left(\frac{a}{b}\right) \left(\frac{b}{a+n^{1/3} \o}\right)-n \left(\frac{a}{2b}\right) \left(\frac{b}{a+n^{1/3}\o}\right)^2+dn^{2/3} \left(\frac{b}{a+n^{1/3} \o}\right). \label{term 3}
	\end{multline}
	Adding (\ref{term 1}) and (\ref{term 3}) together yields
	\begin{multline}h_n(\o) \sim_{\delta}  -n^{2/3} \frac{a}{\o}+n^{1/3} \frac{a^2}{2\o^2}+n \left(\frac{a}{b}\right) \left(\frac{b}{a+n^{1/3} \o}\right) \\ 
	-n \left(\frac{a}{2b}\right) \left(\frac{b}{a+n^{1/3}\o}\right)^2+dn^{2/3} \left(\frac{b}{a+n^{1/3} \o}\right). \label{all terms} \end{multline}
	Adding the first and third terms from (\ref{all terms}) gives the following cancellation.
	\begin{multline*}
	-n^{2/3} \frac{a}{\o}+n \left(\frac{a}{b}\right) \left(\frac{b}{a+n^{1/3} \o}\right)= \\ 
	-n^{2/3} \frac{a}{\o}+n^{2/3} \frac{a}{\o}\left[1-\frac{a}{n^{1/3} \o}+\sum_{k=2}^{\infty}(-1)^k\left(\frac{a}{n^{1/3} \o }\right)^k\right] \sim_{\delta} -n^{1/3}\frac{ a^2}{\o^2},
	\end{multline*}
	thus
	$$h_n(\o) \sim_{\delta} -n^{1/3}\left(\frac{ a^2}{2\o^2} \right)-n \left(\frac{a}{2b}\right) \left(\frac{b}{a+n^{1/3}\o}\right)^2+dn^{2/3} \left(\frac{b}{a+n^{1/3} \o}\right).$$
	When we expand $\frac{b}{a+n^{1/3} \o}=\frac{b}{n^{1/3} \o}+\left(\frac{b}{n^{1/3} \o}\right)\sum_{k=1}^{\infty}\left(\frac{-a}{n^{1/3} \o}\right)^k,$ we see that because $n^{1/3} \o \sim_{\delta} n^{\delta}\i c(a+b)$, the sum is of order $1/c$ times the first term. So we can take only the first terms in our expansion, just as when we Taylor expand. This approximation leads the $n^{2/3}$ terms to cancel giving
	$$ h_n(\o) \sim_{\delta} -n^{1/3} \left(\frac{a^2+ab}{2\o^2}\right) +dn^{1/3} \left(\frac{b}{ \o}\right) \sim_{\delta} n^{1/3}\left(f_1(\o)-t\o\right).
	$$
	
	This implies that $\Re[n^{1/3}t \o +h_n(\o)] \sim_{\delta} \Re[n^{1/3}f_1(\o)].$ Completing the first $\delta$-equivalence in the statement of Lemma \ref{delta bound}.
	
	Now observe that in
	$$\Re[n^{1/3}f_1(\o)]=\Re\left[n^{1/3}\left(t\o -\frac{a(a+b)}{2\o^2}+\frac{bd}{\o}\right)\right],$$
	we can bound the first term $|\Re[n^{1/3}t\o]| \leq n^{\delta}M$. We can bound the third term by $\Re \left[n^{1/3}\frac{bd}{\o}\right] \leq M\frac{n^{2/3-\delta}}{c}$. For the second term, we have $\left|\frac{a(a+b)}{2\o^2}\right|\sim_{\delta} \left(\frac{a(a+b)}{2}\right)\left(\frac{n^{1-2\delta}}{c}\right).$ Thus 
	$$\Re[n^{1/3}f_1(\o)] \sim_{\delta} \left(\frac{a(a+b)}{2}\right)\left(\frac{n^{1-2\delta}}{c}\right).$$
	This gives the second $\delta$-equivalence in the statement of Lemma \ref{delta bound}, and completes the proof. 
\end{proof}

\subsection{Construction of the contour $\mathcal{C}_n$} 

To construct the contour $\mathcal{C}_n$ we will start with lines departing from $\lambda$ at angles $e^{\pm 2 \pi \i/3}$, and with a vertical line $-n^{1/3}a+\i \mathbb{R}$. We will cut both these infinite contours off at specific values $q$ and $p$ respectively which allow us to use our estimates from the previous section on these contours. We will then connect these contours using the level set $\{z: Re[-f_1(z)]=-f_1(\lambda)-\epsilon\}$. The rest of this section is devoted to finding the values $p$ and $q$, showing that our explanation above actually produces a contour, and controlling the derivative of $f_1$ on the vertical segment near $0$. 

We note 
\begin{equation}
f_1(\lambda)=3t^{2/3} \left(\frac{a(a+b)}{2}\right)^{1/3}>0, \label{positive}
\end{equation}
and let
\begin{equation} p=\sqrt{\frac{1}{3}\left(\frac{a(a+b)}{2t}\right)^{2/3}}>0. \label{p definition}
\end{equation}

\noindent By simple algebra, we see that $\Re[-f_1(\pm \i y)]< \Re[-f_1(\lambda)]<0$, when $y < p$, with equality at $y=p$. 

\begin{lemma}\label{derivative}
	$\frac{d}{dy} \Re[-f_1(n^{-1/3}a+\i y)]$ is positive for $y \in [n^{-1/3}|a+b|, p]$, and negative for $y \in [-n^{-1/3}|a+b|, -p].$
\end{lemma}
\begin{proof}We compute
	\begin{align} \frac{d}{dy} \Re[f_1(n^{-1/3}a+\i y)]=&-\Im(\Re[f_1(n^{-1/3}a+\i y)]) \\
	=-\frac{y^3a(a+b)}{|n^{-1/3}a+\i y|^6}&+ \frac{a^2(a+b)n^{-2/3}y}{|n^{-1/3}a+\i y|^6}+\frac{ 3a^2(a+b)b n^{-1/3} y}{2b \lambda |n^{-1/3} a +\i y|^4}. \label{derivative near 0} \end{align}
	Note that for $y \in [n^{-1/3}|a+b|, p] \cup [-n^{-1/3}|a+b|, -p]$, we have $|n^{-1/3} a +\i y| \sim |y|$, so the first term of (\ref{derivative near 0}) is of order $y^{-3}$ and the third term of (\ref{derivative near 0}) is of order $y^{-3}n^{-1/3}$. So for large enough $n$, the third term of (\ref{derivative near 0}) is very small compared to the first term. For $y=\pm n^{-1/3}|a+b|,$ we have $| n^{-1}a(a+b)^4|=|y^3a(a+b)|>|a(a+b)n^{-2/3}ay|=|a^2(a+b)^2n^{-1/3}|$, and the derivative of $y^3a(a+b)$ is larger than the derivative of $a(a+b)n^{-2/3}ay$ for $y \in [n^{-1/3}|a+b|, p] \cup [-n^{-1/3}|a+b|, -p]$, so the first term of (\ref{derivative near 0}) has larger norm than the second term for $y \in [n^{-1/3}|a+b|, p] \cup [-n^{-1/3}|a+b|, -p]$. Thus the sign $\frac{d}{dy} \Re[-f_1(n^{-1/3}a+\i y)]$ is determined by the first term of (\ref{derivative near 0}) in these intervals. 
\end{proof}

Now we can define the contour $\mc{C}_n$. We will give the definition, and then justify that it gives a well defined contour.
\begin{definition}
	Let $q>0$ be a fixed real number such that for $0<y\leq q$, $\frac{d}{dy}\Re[-f_1(\lambda\pm y e^{\pm 2 \pi \i/3})]<0$.
	Let 
	\begin{multline} s=\max\left\lbrace \Re[-f_1(\lambda+ q e^{-2 \pi \i/3})], \Re[-f_1(\lambda+ q e^{ 2 \pi \i/3})], \right. \\  \left . \Re[-f_1(n^{-1/3}(a-\i |a+b|))], \Re[-f_1(n^{-1/3}(a+\i |a+b|))]\right\rbrace.\end{multline}
	Let $\alpha$ be the contourline $\alpha=\{\o:\Re[-f_1(\o)]=s\}$, and define the set
	$$S_n=\{\lambda+ y e^{\pm 2 \pi \i/3}: 0\leq  y \leq q\} \cup \alpha \cup [-an^{-1/3}-\i p,-an^{-1/3}+\i p].$$ 
	For sufficiently large $n$, define the path $\mc{C}_n$ to begin where $\alpha$ intersects $\{\lambda+ y e^{ -2 \pi \i/3}: 0\leq  y \leq q\}$, follow the path $\{\lambda+ y e^{ -2 \pi \i/3}: 0\leq  y \leq q\}$ toward $y=0$, then follow the path $\{\lambda+ y e^{ 2 \pi \i/3}: 0\leq  y \leq q\}$ until it intersects $\alpha.$ $\mc{C}_n$ then follows $\alpha$ in either direction (pick one arbitrarily) until it intersects $[-an^{-1/3}-\i p,-an^{-1/3}+\i p]$ in the upper half plane. $\mc{C}_n$ then follows the path $[-an^{-1/3}-\i p,-an^{-1/3}+\i p]$ toward $-an^{-1/3}- \i p$ until it intersects $\alpha$ in the negative half plane. Then $\mc{C}_n$ follows $\alpha$ in either direction (pick one arbitrarily) until it reaches its starting point where it intersects $\{\lambda+ y e^{ -2 \pi \i/3}: 0\leq  y \leq q\}$. See Figure \ref{Cn picture}\label{Cn def}
\end{definition}
\noindent We see that the $q$ in Definition \ref{Cn def} exists by applying Taylor's theorem along with the fact that $f_1'''(\lambda)>0$, and the $f_1'(\lambda)=f_1''(\lambda)=0$. 

\begin{lemma} \label{N'}
	The sets $\{\lambda+ y e^{ 2 \pi \i/3}: 0\leq  y \leq q\}$ and $\{\lambda+ y e^{ -2 \pi \i/3}: 0\leq  y \leq q\}$ both intersect $\alpha$ at exactly one point. Lemma \ref{N} and Lemma \ref{N'} will show that $\mathcal{C}_n$ is a well defined contour.
\end{lemma}
\noindent This follows from the definition of $q$ and $s$. 
\begin{lemma} \label{N}
	There exists $N>0$ such that for all $n>N$, the sets $[n^{-1/3}+\i n^{-1/3}|a+b|, n^{-1/3}a+p]$ and $[-an^{-1/3}-n^{-1/3}|a+b|, -an^{-1/3}-p]$ both intersect $\alpha$ exactly once.
\end{lemma}
\begin{proof} This is true because
	\begin{equation}
	\Re[-f_1(-n^{-1/3}(a\pm\i |a+b|))]<\Re[-f_1(\lambda)]. \label{small a+b}
	\end{equation}
	by the contour lines in Figure \ref{contour}. This in addition to Lemma \ref{derivative}, and (\ref{positive}) implies the lemma.
\end{proof}

\begin{figure} 
\begin{center}
		\includegraphics[width=11cm]{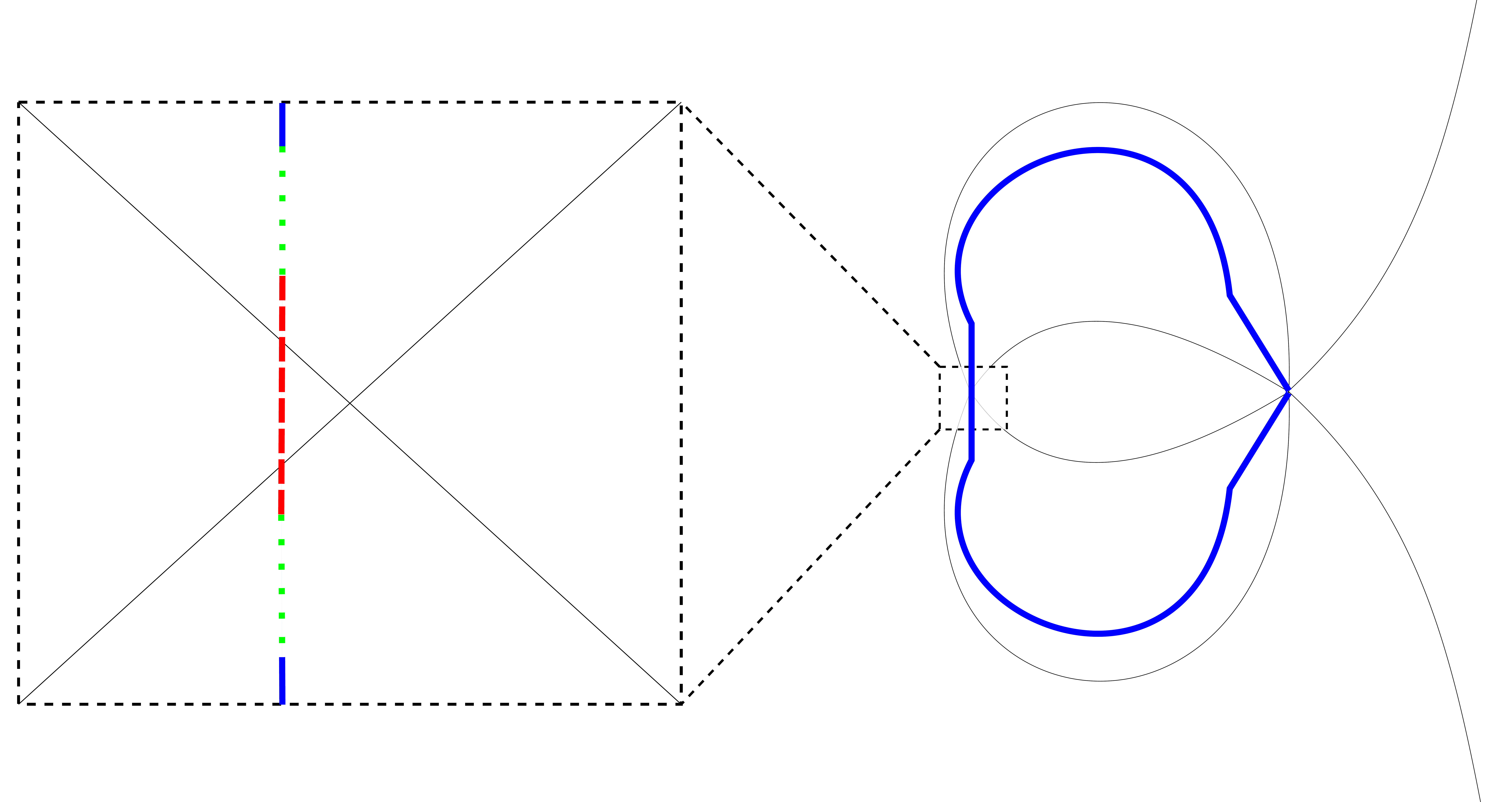}
\end{center}
	\caption{$\mc{C}_n$ is the thick, colored piecewise smooth curve, the contour lines $\{z:\Re[-f_1(z)]=f_1(\lambda)\}$ are the thin black curves. On the right side of the image we see $\mc{C}_n$ as a thick blue curve sandwiched between the contour lines. On the left we zoom in near 0 and see $\mc{C}_n$ pass the real axis as a dotted line to the left of zero. The contour lines meet at the point $0$ on the left and $\lambda$ on the right. We will now describe what section of the proof of Theorem \ref{C bound} bounds $h_n(z)-h_n(\o)+nt^{1/3}(z-\o)$ on different portions of $\mc{C}_n$. The diagonal segments of $\mc{C}_n$ near $\lambda$ are bounded in (ii). The curved segments in the right image, and the solid dark blue vertical segments at the top and bottom of the left image are bounded in (i). The dark red dashed segment that crosses the real axis in the left image is distance $O(n^{-1/3})$ from $0$ and is bounded in (iii). The green dotted segments in the left image are distance $O(n^{\delta-1/3})$ from $0$ for $\delta \in (0,1)$ and are bounded in (iv).}
	\label{Cn picture}
\end{figure}

\subsection{Properties of the contour $\mathcal{C}_n$: proof of Lemma \ref{C bound}}

Most of the work is used to prove part (c). The idea of this proof is to patch together the different estimates from the beginning of Section \ref{Cn}. Away from $0$ we use Lemma \ref{chaos control} and the fact that the contour is steep descent near $\lambda$. Very near $0$ on the scale $n^{-1/3}$ we use Lemma \ref{one third bound}. Moderately near $0$ we use Lemma \ref{delta bound}, and our control of the derivative of $f_1$ on the vertical strip of $\mathcal{C}_n$ near $0$. This last argument allows us to get bounds uniform in $\delta \in (0,1/3)$ when $\o$ is on the scale $n^{1/3-\delta}$. 

\begin{proof}[Proof of Lemma \ref{C bound}]
	
	(a) and (b) follow from the definition of $\mathcal{C}_n$. By a slight modification of the proof of Lemma $2.8$, we see that for $z \in \mc{\gamma}_r$,
	\begin{equation} \Re[h_n(z)-h_n(\lambda)+n^{1/3}t(z -\lambda) \leq n^{1/9}C, \label{gamma lambda}\end{equation}  so to show (c) it suffices to show that for $\o \in \mathcal{C}_n\setminus \mathcal{C}_n^{\epsilon}$, we have
	\begin{equation} \Re[h_n(\lambda)-h_n(\o)+n^{1/3}t(\lambda-\o)] \leq -n^{-1/3} \eta. \label{C lambda} \end{equation} 
	Below we split the contour into $4$ pieces and bound each separately. See Figure \ref{Cn picture}.
	\begin{itemize}
		
		\item[(i)] By Lemma \ref{derivative} and the construction of $\mathcal{C}_n$, we have $\Re[-f_1(\o)] \leq s< \Re[-f_1(\lambda)]$ for $\o \in \mc{C}_n \setminus (\{\lambda+ y e^{\pm 2 \pi \i/3}: 0\leq  y \leq q\} \cup [n^{-1/3}(-a-\i |a+b|),n^{-1/3}(-a+\i |a+b|)])$. So we can apply Lemma \ref{chaos control} and the fact that $f_2$ is bounded outside a neighborhood of $0$ to show that for any $c_1<0$, we have $\Re[h_n(z)-h_n(\lambda)+n^{1/3}t(z-\lambda)] \leq -n^{-1/3} \eta$ for $ \o \in \mc{C}_n \setminus (\{\lambda+ y e^{\pm 2 \pi \i/3}: 0\leq  y \leq q\} \cup [-n^{-1/3}a-\i c_1 |a+b|,-n^{-1/3}a+\i c_1 |a+b|]).$
		
		\item[(ii)]By the definition of $q$, The contour $\{\lambda+ y e^{\pm 2 \pi \i/3}: 0\leq  y \leq q\}$ is steep descent with respect to  the function $f_1$ at the point $\lambda$, so we can apply Lemma \ref{chaos control} and the fact that $f_2$ is bounded outside a neighborhood of $0$ to show $\Re[h_n(z)-h_n(\lambda)+n^{1/3}t(z-\lambda)] \leq -n^{-1/3} \eta$ for $ \o \in \{\lambda+ y e^{\pm 2 \pi \i/3}: 0\leq  y \leq q\} \setminus B_{\epsilon}(\lambda).$

		\item[(iii)] By Lemma \ref{one third bound}, for any $c_0$, we have $\Re[h_n(z)-h_n(\lambda)+n^{1/3}t(z-\lambda)] \leq -n^{-1/3} \eta$ for all $\o \in [n^{-1/3}(-a-\i c_0|a+b|),n^{-1/3}(-a-\i c_0|a+b|)]$. 
		
		\item[(iv)] Now we bound the $\Re[h_n(z)-h_n(\lambda)+n^{1/3}t(z-\lambda)]$ on the last piece of our contour $[n^{-1/3}(-a-\i c_0|a+b|),-n^{-1/3}a+\i c_1 |a+b|] \cup [-n^{-1/3}a-\i c_1 |a+b|,n^{-1/3}(-a-\i c_0|a+b|)].$ We will do this by fixing a constant $c>c_1$, and bounding the function on $\o=n^{-1/3} a + \i n^{\delta-1/3} c (a+b)$ for all pairs $n>N,\delta \in (0,1/3)$ such that $n^{1/3} \leq c_1/c$. 
		
		By Lemma \ref{delta bound}, we have that when $\o= n^{-1/3} a + \i n^{\delta-1/3} c (a+b)$, there exist constants $M_1,M_2,M_3$, such that
		$$\Re[n^{1/3}t \o +h_n(\o)-n^{1/3}f_1(\o)] \leq \frac{n^{2/3-2\delta}}{c^2}M_1+\frac{n^{1-3\delta}}{c^3}M_2+\frac{n^{4/9-\delta}}{c}M_3,$$
		and
		$$f_1(\o) \sim_{\delta} M \frac{n^{1-2\delta}}{c^2}.$$
		First we consider the case when $\delta \in (0,1/3-\epsilon)$. In this case, for any $r>0$ we can choose $c$ and $N_{r}$ large enough that for all $n>N_{r}$,
		$$\frac{\frac{n^{2/3-2\delta}}{c^2}M_1+\frac{n^{1-3\delta}}{c^3}M_2+\frac{n^{4/9-\delta}}{c}M_3}{\Re[n^{1/3}f_1(\o)]}<r/2,$$
		uniformly for all $\delta \in (0,1/3-\epsilon).$ In this case we also have that, by Lemma \ref{chaos control},
		$$|\Re[n^{1/3}t z+h_n(z)]| \leq n^{1/3}f_1(\lambda)+n^{1/9}f_2(\lambda)+C.$$
		By potentially increasing $N_r$, we have that for all $n>N_r$
		$$\frac{|\Re[n^{1/3}t z+h_n(z)]|}{\Re[n^{1/3}f_1(\o)]} \leq r/2.$$
		By Lemma \ref{derivative} and (\ref{small a+b}), for all pairs $n,\delta$ such that $n^{\delta-1/3}<c/c_1$, there is an $\eta>0$ such that
		$$ \Re[-f_1(\o)] \leq \Re[-f_1(\lambda)]-2\eta <-2\eta.$$
		setting $r=1/2$ gives 
		$$\Re[n^{1/3}t (z-\o)+h_n(z)-h_n(\o)] \leq \Re[-n^{1/3} f_1(\o)]+\frac{1}{2} \Re[n^{-1/3}f_1(\o)]<-\eta n^{1/3}.$$
		
		Now we prove the case $\delta \in (1/3-\epsilon, 1/3)$. Note that in the expression
		$$\Re[n^{1/3}t \o +h_n(\o)-n^{1/3}f_1(\o)] \leq \frac{n^{2/3-2\delta}}{c^2}M_1+\frac{n^{1-3\delta}}{c^3}M_2+\frac{n^{4/9-\delta}}{c}M_3,$$
		when $n$ is sufficiently large, we can bound the right hand side by $(M_1+M_2)n^{3\epsilon} \leq (r/2) n^{1/3}$ for any $r>0$. We also have 
		$$|\Re[n^{1/3}t \lambda -h_n(\lambda)-n^{1/3}f_1(\lambda)]| \leq n^{1/9} f_1(\lambda)+C \leq (r/2)n^{1/3}.$$
		The first inequality comes from Lemma \ref{chaos control}, and the second holds for large enough $n$. 
		By Lemma \ref{derivative} and (\ref{small a+b}), for all pairs $n,\delta$ such that $n^{\delta-1/3}<c/c_1$, there is an $\eta>0$ such that
		$$ \Re[-f_1(\o)] \leq \Re[-f_1(\lambda)]-2\eta <-2\eta.$$
		Setting $r=\eta$ gives 
		$$\Re[n^{1/3}t(\lambda-\o)+h_n(\lambda)-h_n(\o)] \leq n^{1/3}\Re[f_1(\lambda)-f_1(\o)] +n^{1/3} \eta \leq -\eta n^{1/3}.$$
	\end{itemize}
	The $c_1$ in part $(i)$ can be chosen as small as desired, the $c$ in part $(iv)$ has already been chosen, and the  $c_0$ in part $(iv)$ can be chosen as large as desired. Choose $c_1<c<c_0$ to complete the proof of (c). 
	
	Given inequalities (\ref{gamma lambda}) and (\ref{C lambda}), part (d) follows if we can show 
	$$\Re[n^{1/3}t(\lambda-\o)+h_n(\lambda)-h_n(\o)],$$
	for $\o \in \mathcal{C}_n^{\epsilon}.$
	Indeed this follows from Lemma \ref{chaos control} and the fact that the contour $\{\lambda+ y e^{\pm 2 \pi \i/3}: 0\leq  y \leq q\}$ is steep descent with respect to  the function $\Re[-f_1]$ at the point $\lambda$.
	
\end{proof}

\section{Dominated convergence} \label{Dominated Convergence}

In this section we carefully prove that the series expansion for $\det(1-\K_n)_{L^2(\mathcal{C}_n^{\epsilon})}$ gives an absolutely convergent series of integrals bounded uniformly in $n$. This allows us to use dominated convergence when we localize the integral in Proposition \ref{cut}, and again when we approximate the kernel by its Taylor expansion in Proposition \ref{taylor approximation}. First we zoom in on a ball of radius epsilon and show that we can absolutely bound $\det(1-\K_n^{\epsilon})_{L^2(\mathcal{C}_n^{\epsilon})}$ uniformly in $n$. 
\begin{lemma} \label{dom}
	For any sufficiently small $\epsilon>0$, and sufficiently large $r$, there exists a function $\overline{F}(\overline{\o},\overline{\o}')$, such that for all $\overline{\o}, \overline{\o}' \in \mathcal{C}_{-1}^{n^{1/9} \epsilon}$, $z \in \mathcal{D}_0^{n^{1/9} \epsilon}$, $n>N$ the integrand of $\overline{\mathsf{K}}_{n}^{\epsilon}(\overline{\o},\overline{\o}')$ in equation (\ref{K def}) is absolutely bounded by $\overline{F}(\overline{\o},\overline{\o}',\overline{z})$, and
	
	\begin{equation}\sum_{m=0}^{\infty}\int_{(\mathcal{C}_{-1}^{n^{1/9} \epsilon})^m} \left|\det \left(\int_{\mathcal{D}_0^{n^{1/9} \epsilon}}\overline{F}(\overline{\o}_i,\overline{\o}_j,\overline{z}) d\overline{z}\right)_{i,j=1}^m \right| d\overline{\o}_1...d\overline{\o}_m<\infty. \label{bar dominated} \end{equation}
\end{lemma}

\begin{proof}
	For $\overline{\o},\overline{\o}' \in \mathcal{C}_{-1}^{\epsilon}$, and $\overline{z} \in \mathcal{D}_0^{\epsilon}$, we have
	
	$$\left|\frac{\lambda+n^{-1/9} \overline{z}}{\lambda+n^{-1/9} \overline{\o}}\right| \leq \left|\frac{\lambda+\epsilon}{\lambda-\epsilon}\right|,$$
	and by Taylor approximation, we have the additional bounds
	
	\begin{align} n^{1/3} (f_1(\lambda+n^{-1/9}\overline{z})-f_1(\lambda+n^{-1/9} \overline{\o})) &\leq (f_1'''(\lambda)+\delta_1)(\overline{z}^3-\overline{\o}^3),
	\label{f_1} \\
	n^{1/9}(f_2(\lambda+n^{-1/9} \overline{z})-f_2(\lambda+n^{-1/9} (\overline{\o}))) &\leq (f_2'(\lambda)+\delta_2)(\overline{z}-\overline{\o}),
	\label{f_2}\\
	r_n(\lambda+n^{-1/9} \overline{z})-r_n(\lambda+n^{-1/9} \overline{\o}) &\leq C n^{-1/9} (\overline{z}-\overline{\o}) \leq C\epsilon \leq \delta_3.
	\label{f_3}\end{align}
	
	Note that in these bounds we can make $\delta_1,\delta_2,\delta_3$ as small as desired by choosing $\epsilon$ small. Equations (\ref{f_1}) and (\ref{f_2}) follow from the fact that $f_1$, and $f_2$ are holomorphic in the compact set $\overline{B}_{\epsilon}(\lambda)$. And equation (\ref{f_3}) follows from Lemma \ref{chaos control}. Note that along $\mathcal{D}_0$, $z$ is purely imaginary, so (\ref{f_1}),(\ref{f_2}), and (\ref{f_3}) show that the full exponential in the integrand in (\ref{K def}) is bounded above by
	
	\begin{equation} e^{2 \delta_3} e^{-(f_1'''(\lambda)-\delta_1)\overline{\o}^3-(f'_2(\lambda)-\delta_2)\overline{\o}}. \label{omega bound}\end{equation}
	We choose $\epsilon$ small enough that $\delta_1 < f'''_1(\lambda)$, so that (\ref{omega bound}) has exponential decay as $\o$ goes to $\infty$ in directions $e^{\pm 2 \pi \i/3}$. Set
	$$\overline{F}(\overline{\o},\overline{\o}',\overline{z})=\left|\left(\frac{\lambda+\epsilon}{\lambda-\epsilon}\right) e^{2\delta_3}e^{-(f_1'''(\lambda)-\delta_1)\overline{\o}^3-(f'_2(\lambda)-\delta_2)}\frac{1}{(\overline{z}+1)(\overline{z}+1)}\right|.$$
	By the sentence preceeding (\ref{omega bound}) $\overline{F}$ absolutely bounds the integrand of $\overline{\mathsf{K}}_{n}^{\epsilon}$.
	Now set $L_1= \frac{|\lambda+\epsilon|}{|\lambda-\epsilon|} e^{2\delta_3} \int_{\mathcal{D}_0} \frac{1}{(\overline{z}+1)(\overline{z}+1)} d \overline{z}$ so that $2 e^{2\delta_3}\int_{\mathcal{D}_0} \frac{1}{(\overline{z}-\overline{\o})(\overline{z}-\overline{\o}')} d \overline{z} \leq L_1.$ Then
	
	\begin{equation}\int_{\mathcal{D}_0^{\epsilon}}\overline{F}(\overline{\o},\overline{\o}',\overline{z}) \leq L_1 \left| e^{-(f_1'''(\lambda)-\delta_1)\overline{\o}^3-(f'_2(\lambda)-\delta_2)}\right| , \label{L}
	\end{equation}
	
	By Hadamard's bound 
	
	$$\left|\det\left(\int_{\mathcal{D}_0^{n^{1/9}} \epsilon}\overline{F}(\overline{\o}_i,\overline{\o}_j',\overline{z}) d \overline{z}\right)_{i,j=1}^m\right| \leq m^{m/2} L_1^m \prod_{i=1}^m \left| e^{-(f'''_1(\lambda)-\delta)\overline{\o}^3-(f_2'(\lambda)-\delta)\overline{\o}} \right|.$$
	
	Now  because $\delta_1<f'''_1(\lambda)$, we can set
	
	$$S=\int_{\mathcal{C}_{-1}^{n^{1/9} \epsilon}} \left| e^{-(f'''_1(\lambda)-\delta)\overline{\o}^3-(f_2'(\lambda)-\delta)\overline{\o}} \right| d \overline{\o}<\infty.$$
	
	Then we have the bound,
	
	$$\int_{(\mathcal{C}_{-1}^{n^{1/9}\epsilon})^m} \left|\det\left(\int_{\mathcal{D}_0^{n^{1/9}\epsilon}} \overline{F}(\overline{\o}_i,\overline{\o}_j',\overline{z}) d\overline{z}\right)_{i,j=1}^m\right| d\overline{\o}_1...d\overline{\o}_m \leq m^{m/2} (SL_1)^m.$$
	
	So by Stirling's approximation
	
	$$\sum_{m=0}^{\infty}\int_{(\mathcal{C}_{-1}^{n^{1/9} \epsilon})^m} \left|\det\left(\int_{\mathcal{D}_0^{n^{1/9} \epsilon}} \overline{F}(\overline{\o}_i,\overline{\o}_j,\overline{z}) d \overline{z}\right)_{i,j=1}^m\right| d\overline{\o}_1...d\overline{\o}_m< \infty.$$
\end{proof}

The next lemma completes our dominated convergence argument, by controlling the contribution to $\det(I-K_n)_{L^2(\mc{C}_n^{\epsilon})}$ of $z \in \gamma_r \setminus \gamma_r^{\epsilon}$.

\begin{lemma} \label{dom 2}
	
	For any sufficiently small $\epsilon>0$, and sufficiently large $r$, there is a function $\overline{G}(\overline{\o},\overline{\o}',\overline{z})$, and a natural number $N$, such that for all $\overline{\o},\overline{\o}' \in \overline{\mathcal{C}}_{n}^{\epsilon}$ and $\overline{z} \in \overline{\mathcal{\gamma}}_r$, $n >N$, the integrand of $\overline{\mathsf{K}}_{n}(\overline{\o},\overline{\o}')$ is absolutely bounded by $\overline{G}(\overline{\o},\overline{\o}',\overline{z})$, and
	
	\begin{equation} \sum_{m=0}^{\infty} \frac{1}{m!} \int_{(\overline{C}^{\epsilon})^m} \left| \det \left(\int_{\overline{\mathcal{\gamma}}_r}\overline{G}(\overline{\o}_i,\overline{\o}_j,\overline{z})dz \right)_{i,j=1}^m\right| d\overline{\o}_i...d\overline{\o}_j<\infty, \label{dominated} \end{equation}
	
	where $\overline{\mathcal{\gamma}}_r$ and $\overline{C}_n^{\epsilon}$ are the rescaled contours of $\mc{\gamma}_r$ and $C_n^{\epsilon}$ respectively.
\end{lemma}
\begin{proof}
	
	Let $\overline{G}=\overline{F}$ for $z \in \mathcal{\gamma}_r^{\epsilon}$. We decompose the integral along $\mathcal{\gamma}_r$ in three parts:  the integral along $\mathcal{\gamma}_r^{\epsilon}$, the integral along $(e^{-2\pi \i/3} \infty,-r) \cup (r,e^{2\pi \i/3} \infty)$ and the integral along $[-r,-\epsilon] \cup [\epsilon,r]$. For $z \in \mathcal{\gamma}_r \setminus \mathcal{\gamma}_r^{\epsilon}$ we have the following bounds
	
	\begin{align}|e^{n^{1/3}t(z-\o) +h_n(z)-h_n(\o)}| &\leq  |e^{n^{1/3}(f_1(z)-f_1(\o))+n^{1/9}C_2+C_3}|\nonumber \\
	&\leq |e^{n^{1/3}(f_1(z)-f_1(\o)+\delta)}| \nonumber \\
	&\leq |e^{n^{1/3}(f_1(z)-f_1(\lambda)+\delta)}||e^{n^{1/3}(f_1(\lambda)-f_1(\o))}| \label{exp}.
	\end{align}
	
	Where the first inequality follows from Lemma \ref{chaos control}. If we choose $\delta< \eta/2$, and recall that if $z \in \mathcal{\gamma}_r \setminus \mathcal{\gamma}_r^{\epsilon}$, then $f_1(z)-f_1(\lambda)< -\eta,$ so $f_1(z)-f_1(\lambda)+\delta< -\eta/2<0$. So if we wish we can bound (\ref{exp}) by either of the following expressions
	
	\begin{equation}
	|e^{n^{1/3}(f_1(\lambda)-f_1(\o))}| \label{exp 2}
	\end{equation}
	\begin{equation}
	|e^{n^{1/9}(-tz+t\lambda)}||e^{n^{1/3}(f_1(\lambda)-f_1(\o))}| \label{exp 3}
	\end{equation}
	
	The bound (\ref{exp 3}) follows from the fact that we can choose $r$ large enough so that $|f_1(z)+tz| \leq \delta$ outside $B_r(0)$. Then because the exponent in the first factor of (\ref{exp}) is negative, for large enough $n$ we can remove the constant $\delta$ in return for reducing $n^{1/3}$ to $n^{1/9}$. 
	
	Now for $z \in [-r,-\epsilon] \cup [\epsilon,r]$, we have
	
	$$\left|\frac{z}{\o}\right| \leq \left| \frac{r+\lambda}{\lambda-\epsilon}\right|, \qquad \left|\frac{1}{(\overline{z}-\overline{\o})(\overline{z}-\overline{\o}')}\right| \leq 1.$$
	
	So for $z \in [-r,-\epsilon] \cup [\epsilon,r]$, we set 
	
	$$\overline{G}(\overline{\o},\overline{\o}',\overline{z})=\left| \frac{r+\lambda}{\lambda-\epsilon} \right| \left|  \frac{1}{(\overline{z}-\overline{\o})(\overline{z}-\overline{\o}')} \right|  \left| e^{n^{1/3}(f_1(\lambda)-f_1(\o))}\right|.$$
	
	Using the above bounds and (\ref{exp 2}) we see that the integrand of $\overline{\mathsf{K}}_n$ is absolutely bounded by $\overline{G}$ in this region. Set $L_2=\int_{i\mathbb{R}}\frac{r+\lambda}{\lambda-\epsilon}\frac{1}{(\overline{z}+1)(\overline{z}+1)} d \overline{z}$ so that the integral of $\overline{G}$ on the rescaled contour of $[-r,-\epsilon] \cup [\epsilon,r]$ is bounded by $L_2|e^{n^{1/3}(f_1(\lambda)-f_1(\o))}|$. 
	
	For $z \in (e^{-2\pi \i/3} \infty,-r) \cup (r,e^{2\pi \i/3} \infty)$, we have
	
	$$\left|\frac{1}{(\overline{z}-\overline{\o})(\overline{z}-\overline{\o}')}\right| \leq 1.$$
	
	So for $z \in (e^{-2\pi \i/3} \infty,-r) \cup (r,e^{2\pi \i/3} \infty)$, we set
	
	$$\overline{G}(\overline{\o},\overline{\o}',\overline{z})= \left|\frac{z}{\o} \right| \left| e^{t(\lambda-\overline{z})} \right| \left| e^{(-f_1'''(\lambda)+\delta)\overline{\o}}\right|.$$
	
	Thus by (\ref{exp 3}), we can see that the integrand of $\overline{\mathsf{K}_n}$ is absolutely bounded by $\overline{G}$ in this region. Now let $L_3=\int_{(e^{-2 \pi \i/3} \infty,-r] \cup [r,e^{2\pi \i/3} \infty)} \left|\frac{\lambda+\overline{z}}{\lambda-\epsilon}\right| |e^{t(\lambda-\overline{z})}| d \overline{z}$. For all $n$, the integral of $\overline{G}$ over the rescaled contour $(e^{-2 \pi \i/3} \infty,-r] \cup [r,e^{2\pi \i/3} \infty)$ is bounded above by $L_3 |e^{(-f_1'''(\lambda)+\delta)\overline{\o}^{3}}|$. 
	
	Let $\overline{\mathcal{\gamma}}_r$ be the rescaled contour $\mathcal{\gamma}_r$ in the variable $\overline{z}$
	
	\begin{equation}
	\int_{\overline{\mathcal{\gamma}}_r} \overline{G} d\overline{z} \leq (L_1+L_2+L_3) e^{(-f_1'''(\lambda)+\delta)\overline{\o}^{3}} \leq L e^{(-f_1'''(\lambda)+\delta)\overline{\o}^{3}}, \label{G bound}
	\end{equation}
	where the constant $L$ comes from (\ref{L}). Thus we have bounded $\int_{\overline{\mathcal{\gamma}}_r} \overline{G} d\overline{z}$ by a constant times a term which has exponential decay as $\overline{\o} \to e^{\pm 2\pi \i/3} \infty$. The same argument as in Lemma \ref{dom} shows that 
	
	$$\sum_{m=0}^{\infty} \frac{1}{m!} \int_{(C^{\epsilon})^m} \left|\det\left(\int_{\mathcal{\gamma}_r^{\epsilon}}G(\o_i,\o_j,z)dz\right)_{i,j=1}^m\right| d\o_i...d\o_j<\infty.$$
	
\end{proof}

\begin{lemma} \label{max bound lemma} Let $\o_1 \in \mc{C}_n \setminus \mc{C}_n^{\epsilon}$ and $\o_2,..,\o_m \in \mc{C}^n$. There exist positive constants $M,L_4, \eta>0$ so that for sufficiently large $n$, we have
	\begin{equation}|\overline{\mathsf{K}}_{n}(\overline{\o}_i,\overline{\o}_j)| \leq M \qquad  \nonumber \end{equation}
	and
	\begin{equation} |\overline{\mathsf{K}}_{n}(\overline{\o}_1, \overline{\o}_i)| \leq L_4 n^{4/9} e^{-n^{1/3} \eta}, \nonumber \end{equation}
	for all $i,j$.
\end{lemma}

\begin{proof}
	By Lemma \ref{C bound}, for any $\epsilon>0$, there exists a $N,C>0$, such that if $v \in \mc{C}_{n}\setminus \mc{C}_n^{\epsilon}$, and $z \in \mathcal{\gamma}_r$, then for all sufficiently large $n$, we have
	$$\Re[h_n(z)-h_n(\o)+n^{1/3}t(z-\o)] \leq  -n^{1/3} \eta.$$
	For $z \in \mathcal{\gamma}_r$ and $\o, \o' \in \mc{C}_n \setminus \mc{C}_n^{\epsilon}$, $n>N$ we have the following bounds:
	$$
	\frac{1}{(z-\o)(z-\o')} \leq \left(\frac{2}{\epsilon}\right)^2, \qquad \frac{1}{\o} \leq \frac{n^{1/3}}{a},
	$$
	and
	\begin{align}
	|e^{n^{1/3}t(z-\o)+h_n(z)-h_n(\o)}| &\leq |e^{n^{1/3}(f_1(z)-f_1(\o)+\delta)}| \label{exponent 1} \\
	&\leq |e^{n^{1/3}(f_1(z)-f_1(\lambda)}||e^{n^{1/3}(f_1(\lambda)-f_1(\o)+\delta)}| \label{exponent 2}
	\end{align}
	where (\ref{exponent 1}) follows from (\ref{chaos control}) and the fact that $f_2$ is bounded away from $0$. Note that for $z \in \mathcal{\gamma}_r$, $|f_1(z)-f_1(\lambda)| \leq 0$, and for $\o, \o' \in \mc{C}_n \setminus \mc{C}_n^{\epsilon}$, $f_1(\lambda)-f_1(\o) + \delta< -\eta$, so (\ref{exponent 2}) is bounded above by
	$$|e^{(f_1(z)-f_1(\lambda)}||e^{-n^{1/3}\eta}|.$$
	Thus if we set $L_4=\frac{2^2}{a \epsilon^2} \int_{\mathcal{\gamma}_r} |z| |e^{f_1(z)-f_1(\lambda)}| dz< \infty$, we get
	$$|\mathsf{K}_n(\o,\o')| \leq  L_4 n^{1/3} e^{-n^{1/3} \eta}.$$
	So if we change the variable of integration to $d\overline{z}=n^{1/9} d z$ gives. 
	\begin{equation} 
	|\overline{\mathsf{K}}_{n}(\overline{\o}, \overline{\o}')| \leq L_4 n^{4/9} e^{-n^{1/3} \eta} \qquad  \text{for $\o,\o' \in \mc{C}_n \setminus \mc{C}_n^{\epsilon}$} \label{outer bound}
	\end{equation}
	Let $\o_1 \in \mc{C}_n \setminus \mc{C}_n^{\epsilon}$ and $\o_2,..,\o_m \in \mc{C}^n$, then for $i \neq 1$,
	\begin{align}|\overline{\mathsf{K}}_{n}(\overline{\o}_1, \overline{\o}_i)| &\leq L_4 n^{4/9} e^{-n^{1/3} \eta},\nonumber \\
	|\overline{\mathsf{K}}_{n}(\overline{\o}_i,\overline{\o}_j)| &\leq \max[Le^{(-f_1'''(\lambda)+\delta) \overline{\o}^{3}}, L_4 n^{4/9} e^{-n^{1/3} \eta}] \leq M.\end{align}
	The first equality follows from (\ref{G bound}) and the second inequality holds for large $n$, when we set $M=\max[L_4,L]$ because $-f_1'''(\lambda)+\delta<0$.
\end{proof}

The last thing we need to complete the proof of Theorem \ref{main theorem} is to bound (\ref{sum bound}) from Proposition (\ref{localize}). We do so in the following lemma.
\begin{lemma} \label{sum} For any $C>1$, we have
	$$\sum_{m=1}^{\infty} \frac{1}{m!} C^m m^{1+m/2} \leq 16 C^4 e^{2C^2}.$$
\end{lemma}
\begin{proof} We have 
	$$\frac{m^{1+m/2}}{m!} \leq \frac{m 2^{m/2}}{(\lfloor m/2 \rfloor)!},$$
	so that 
	\begin{align*}
	\sum_{m=1}^{\infty} \frac{1}{m!} C^m m^{1+m/2} &\leq \sum_{m=1}^{\infty} \frac{m}{(\lfloor m/2 \rfloor)!} (2C^2)^{m/2}\\
	&\leq \sum_{k=1}^{\infty} \frac{2k (2C^2)^k}{k!}+\sum_{k=1}^{\infty} \frac{(2k+1)(2C^2)^{k+1}}{k!} \\ 
	&\leq 16 C^4 e^{2 C^2}.
	\end{align*} 
	
\end{proof}

%
%
%
%
%
%
%
%
%
%

\providecommand{\bysame}{\leavevmode\hbox to3em{\hrulefill}\thinspace}
\providecommand{\MR}{\relax\ifhmode\unskip\space\fi MR }
\providecommand{\MRhref}[2]{%
	\href{http://www.ams.org/mathscinet-getitem?mr=#1}{#2}
}
\providecommand{\href}[2]{#2}

%

\begin{thebibliography}{10}
\bibitem{aggarwal2016current}
Amol Aggarwal, \emph{Current fluctuations of the stationary {ASEP} and
	six-vertex model}, Duke Math J., arXiv:1608.04726 (2016).

\bibitem{aggarwal2017dynamical}
\bysame, \emph{Dynamical stochastic higher spin vertex models}, Selecta Math.
(2017), 1--77.

\bibitem{aggarwal2016phase}
Amol Aggarwal and Alexei Borodin, \emph{Phase transitions in the {ASEP} and
	stochastic six-vertex model}, arXiv preprint arXiv:1607.08684 (2016).

\bibitem{UniversalityRectanglesCorwin}
Antonio Auffinger, Jinho Baik, and Ivan Corwin, \emph{{U}niversality for
	directed polymers in thin rectangles}, arXiv preprint arXiv:1204.4445 (2012).

\bibitem{fppbook}
Antonio Auffinger, Michael Damron, and Jack Hanson, \emph{50 years of
	first-passage percolation}, vol.~68, American Mathematical Soc., 2017.

\bibitem{baik2017facilitated}
Jinho Baik, Guillaume Barraquand, Ivan Corwin, and Toufic Suidan,
\emph{Facilitated exclusion process}, to appear in proceedings of Abel 2016
symposium, arXiv:1707.01923 (2017).

\bibitem{baik1999distribution}
Jinho Baik, Percy Deift, and Kurt Johansson, \emph{On the distribution of the
	length of the longest increasing subsequence of random permutations}, J.
Amer. Math. Soc. \textbf{12} (1999), no.~4, 1119--1178.

\bibitem{baik2005gue}
Jinho Baik and Toufic~M. Suidan, \emph{A {GUE} central limit theorem and
	universality of directed first and last passage site percolation}, Int. Math.
Res. Not. \textbf{2005} (2005), no.~6, 325--337.

\bibitem{balazs-rassoul}
M{\'a}rton Bal{\'a}zs, Firas Rassoul-Agha, and Timo Sepp{\"a}l{\"a}inen,
\emph{Large deviations and wandering exponent for random walk in a dynamic
	beta environment}, arXiv preprint arXiv:1801.08070 (2018).

\bibitem{barraquand2014phase}
Guillaume Barraquand, \emph{A phase transition for q-{TASEP} with a few slower
	particles}, Stochastic Process. Appl. \textbf{125} (2015), no.~7, 2674 --
2699.

\bibitem{barraquand2017stochastic}
Guillaume Barraquand, Alexei Borodin, Ivan Corwin, and Michael Wheeler,
\emph{{Stochastic six-vertex model in a half-quadrant and half-line open
		ASEP}}, Duke Math. J., arXiv:1704.04309 (2017).

\bibitem{barraquand2016q}
Guillaume Barraquand and Ivan Corwin, \emph{The {$q$}-{H}ahn asymmetric
	exclusion process}, Ann. Appl. Probab. \textbf{26} (2016), no.~4, 2304--2356.

\bibitem{RWRE}
\bysame, \emph{{R}andom-walk in {B}eta-distributed random environment}, Prob.
Theory Related Fields \textbf{167} (2017), no.~3-4, 1057--1116.

\bibitem{UniversalPropertyCloseToAxis}
Thierry Bodineau and James Martin, \emph{{A} universality property for
	last-passage percolation paths close to the axis}, Electron. Commun. Probab.
\textbf{10} (2005), 105--112.

\bibitem{borodin2017family}
Alexei Borodin, \emph{On a family of symmetric rational functions}, Adv. Math.
\textbf{306} (2017), 973--1018.

\bibitem{borodin2014macdonald}
Alexei Borodin and Ivan Corwin, \emph{Macdonald processes}, Probab. Theory
Related Fields \textbf{158} (2014), no.~1-2, 225--400.

\bibitem{FreeEnergyCorwin}
Alexei Borodin, Ivan Corwin, and Patrik Ferrari, \emph{{F}ree energy
	fluctuations for directed polymers in random media in 1+ 1 dimension}, Comm.
Pure Appl. Math \textbf{67} (2014), no.~7, 1129--1214.

\bibitem{borodin2015height}
Alexei Borodin, Ivan Corwin, Patrik Ferrari, and B\'alint Vet{\H{o}},
\emph{Height fluctuations for the stationary {KPZ} equation}, Math. Phys.
Anal. Geom. \textbf{18} (2015), no.~1, Art. 20, 95.

\bibitem{borodin2016stochastic}
Alexei Borodin, Ivan Corwin, and Vadim Gorin, \emph{Stochastic six-vertex
	model}, Duke Math. J. \textbf{165} (2016), no.~3, 563--624.

\bibitem{borodin2015spectral}
Alexei Borodin, Ivan Corwin, Leonid Petrov, and Tomohiro Sasamoto,
\emph{Spectral theory for interacting particle systems solvable by coordinate
	{B}ethe ansatz}, Comm. Math. Phys. \textbf{339} (2015), no.~3, 1167--1245.

\bibitem{borodin2013log}
Alexei Borodin, Ivan Corwin, and Daniel Remenik, \emph{Log-gamma polymer free
	energy fluctuations via a fredholm determinant identity}, Comm. Math. Phys.
\textbf{324} (2013), no.~1, 215--232.

\bibitem{borodin2014duality}
Alexei Borodin, Ivan Corwin, and Tomohiro Sasamoto, \emph{From duality to
	determinants for {q-TASEP} and {ASEP}}, Ann. Probab. \textbf{42} (2014),
no.~6, 2314--2382.

\bibitem{BorodinFerrari}
Alexei Borodin and Patrik Ferrari, \emph{{L}arge time asymptotics of growth
	models on space-like paths {I}: Push{ASEP}}, Electron. J. Probab. \textbf{13}
(2008), 1380--1418.

\bibitem{borodin2016asep}
Alexei Borodin and Grigori Olshanski, \emph{The {ASEP} and determinantal point
	processes}, Commun. Math, Phys. \textbf{353} (2017), no.~2, 853--903.

\bibitem{borodin2018higher}
Alexei Borodin and Leonid Petrov, \emph{Higher spin six vertex model and
	symmetric rational functions}, Selecta Math. \textbf{24} (2018), no.~2,
751--874.

\bibitem{chaumont2017fluctuation}
Hans Chaumont and Christian Noack, \emph{Fluctuation exponents for stationary
	exactly solvable lattice polymer models via a {M}ellin transform framework},
arXiv preprint arXiv:1711.08432 (2017).

\bibitem{universality}
Ivan Corwin, \emph{{T}he {K}ardar-{P}arisi-{Z}hang equation and universality
	class}, Random Matrices Theory Appl. \textbf{1} (2012), no.~01, 1130001.

\bibitem{cor14}
Ivan Corwin, \emph{{T}he q-{H}ahn {B}oson process and q-{H}ahn {TASEP}}, Int.
Math. Res. \textbf{2015} (2014), no.~14, 5577--5603.

\bibitem{stoplight}
Ivan Corwin, \emph{{K}ardar-{P}arisi-{Z}hang universality}, Not. {A}{M}{S}
(2016).

\bibitem{corwin-gu}
Ivan Corwin and Yu~Gu, \emph{{Kardar--Parisi--Zhang} equation and large
	deviations for random walks in weak random environments}, J. Stat. Phys.
\textbf{166} (2017), no.~1, 150--168.

\bibitem{corwin2016stochastic}
Ivan Corwin and Leonid Petrov, \emph{Stochastic higher spin vertex models on
	the line}, Comm. Math. Phys. \textbf{343} (2016), no.~2, 651--700.

\bibitem{corwin2014strict}
Ivan Corwin, Timo Sepp{\"a}l{\"a}inen, and Hao Shen, \emph{The strict-weak
	lattice polymer}, J. Stat. Phys. (2015), 1--27.

\bibitem{qtasep}
Patrick Ferrari and B\'alint Vet{\H{o}}, \emph{{T}racy {W}idom asymptotics for
	q-{T}{A}{S}{E}{P}},  \textbf{51} (2015), no.~4, 1465--1485.

\bibitem{fontes2002brownian}
Luiz Fontes, Marco Isopi, Charles Newman, and Krishnamurthi Ravishankar,
\emph{The {B}rownian web}, Proc. Nat. Acad. Sci. \textbf{99} (2002), no.~25,
15888--15893.

\bibitem{fontes2004brownian}
\bysame, \emph{The {B}rownian web: characterization and convergence}, Ann.
Probab. \textbf{32} (2004), no.~4, 2857--2883.

\bibitem{ghosalpush}
Promit Ghosal, \emph{{H}all-{L}ittlewood-push{T}{A}{S}{E}{P} and its {K}{P}{Z}
	limit}, arXiv preprint arXiv:1701.07308 (2017).

\bibitem{hammersley1965first}
John Hammersley and Dominic Welsh, \emph{First-passage percolation, subadditive
	processes, stochastic networks, and generalized renewal theory}, Bernoulli
1713, Bayes 1763, Laplace 1813, Springer, 1965, pp.~61--110.

\bibitem{johansson2000shape}
Kurt Johansson, \emph{Shape fluctuations and random matrices}, Comm. Math.
Phys. \textbf{209} (2000), no.~2, 437--476.

\bibitem{originalKPZ}
Mehran Kardar, Giorgio Parisi, and Yi-Cheng Zhang, \emph{Dynamic scaling of
	growing interfaces}, Phys. Rev. Lett. \textbf{56} (1986), 889--892.

\bibitem{krishnan2016tracy}
Arjun Krishnan and Jeremy Quastel, \emph{{T}racy-{W}idom fluctuations for
	perturbations of the log-gamma polymer in intermediate disorder}, arXiv
preprint arXiv:1610.06975 (2016).

\bibitem{newman2010marking}
Charles Newman, Krishnamurthi Ravishankar, and Emmanuel Schertzer,
\emph{Marking (1, 2) points of the {B}rownian web and applications},
\textbf{46} (2010), no.~2, 537--574.

\bibitem{o2014tracy}
Neil O'Connell and Janosch Ortmann, \emph{{T}racy-{W}idom asymptotics for a
	random polymer model with gamma-distributed weights}, Electron. J. Probab.
\textbf{20} (2015), no. 25, 1--18.

\bibitem{orr2017stochastic}
Daniel Orr and Leonid Petrov, \emph{Stochastic higher spin six vertex model and
	q-{TASEP}s}, Adv. Math. \textbf{317} (2017), 473--525.

\bibitem{pov13}
Alexander Povolotsky, \emph{{O}n the integrability of zero-range chipping
	models with factorized steady states}, J. Phys. A \textbf{46} (2013), no.~46,
465205.

\bibitem{prahofer2000universal}
Michael Pr{\"a}hofer and Herbert Spohn, \emph{Universal distributions for
	growth processes in {1+1} dimensions and random matrices}, Phys. rev. lett.
\textbf{84} (2000), no.~21, 4882.

\bibitem{schertzer2015brownian}
Emmanuel Schertzer, Rongfeng Sun, and Jan Swart, \emph{The {B}rownian web, the
	{B}rownian net, and their universality}, Advances in Disordered Systems,
Random Processes and Some Applications (2015), 270--368.

\bibitem{sun2008brownian}
Rongfeng Sun and Jan Swart, \emph{The {B}rownian net}, Ann. Probab. (2008),
1153--1208.

\bibitem{thieryzerotemp}
Thimoth{\'e}e Thiery, \emph{{S}tationary measures for two dual families of
	finite and zero temperature models of directed polymers on the square
	lattice}, J. Stat. Phys. \textbf{165} (2016), no.~1, 44--85.

\bibitem{thiery2015integrable}
Thimoth{\'e}e Thiery and Pierre Le~Doussal, \emph{On integrable directed
	polymer models on the square lattice}, J. Phys. A \textbf{48} (2015), no.~46,
465001.

\bibitem{thiery2016exact}
\bysame, \emph{Exact solution for a random walk in a time-dependent {1D} random
	environment: the point-to-point {B}eta polymer}, J. Phys. A \textbf{50}
(2016), no.~4, 045001.

\bibitem{tracy2008fredholm}
Craig~A. Tracy and Harold Widom, \emph{A {F}redholm determinant representation
	in {ASEP}}, J. Stat. Phys. \textbf{132} (2008), no.~2, 291--300.

\bibitem{tracy2008integral}
\bysame, \emph{Integral formulas for the asymmetric simple exclusion process},
Comm. Math. Phys. \textbf{279} (2008), no.~3, 815--844.

\bibitem{asep}
\bysame, \emph{Asymptotics in {A}{S}{E}{P} with step initial condition}, Comm.
Math. Phys. \textbf{290} (2009), no.~1, 129--154.

\bibitem{qhahnboson}
B{\'a}lint Vet{\H{o}}, \emph{{T}racy-{W}idom limit of {$q$-Hahn}
	{T}{A}{S}{E}{P}}, Electron. J. Probab. \textbf{20} (2015).

\end{thebibliography}

\end{document}